\documentclass[12pt]{amsart}

 \usepackage{graphicx}
 \usepackage{amscd, color}

\usepackage{marginnote}

 \def\ga{\gamma}
 \def\al{\alpha}
 \def\be{\beta}
 \def\de{\delta}
 
 \def\eps{\varepsilon}
 \def\ga{\gamma}
  \def\G{\Gamma}
 
 \def\la{\lambda}

 \def\EE{{\mathbf E}}
 \def\VV{{\mathbf V}}

  \def\G{\Gamma}

 \def\R{{\mathbb R}}

 \def\N{{\mathbb N}}
 \def\T{{\mathbb T}}

\def\< {\langle}
\def\> {\rangle}
 \def\oo{\mathrm o}
 \def\tt{{\mathrm t}}

\def\wtd {\widetilde}

 \DeclareMathOperator{\SC}{\mathcal S}

  \renewcommand{\proofname}{{\bf Proof:}}

 \theoremstyle{plain}

  \swapnumbers

 \newtheorem{Thm}{Theorem}[section]
 
 \newtheorem{Lemma}[Thm]{\bf Lemma}
 
 \newtheorem{Corollary}[Thm]{\bf Corollary}
 \newtheorem{Theorem}[Thm]{\bf Theorem}
 \newtheorem{Proposition}[Thm]{\bf Proposition}
 
 \newtheorem{Notation}[Thm]{\bf Notation}

 \theoremstyle{definition}
 \newtheorem{Definition}[Thm]{\bf Definition}

 \theoremstyle{remark}
 \newtheorem{Remark}[Thm]{\bf Remark}

 \newtheoremstyle{Cl}
  {5pt}
  {3pt}
  {\sl}
  {}
  {\it}
  {:}
  {.5em}
  {}

 \theoremstyle{Cl}

 \def\begincproof{
                  \renewcommand{\proofname}{\it Proof:}
                  \begin{proof}
                 }

 \def\endcproof{
                \renewcommand{\qedsymbol}{$\diamondsuit$}
                \end{proof}
                \renewcommand{\qedsymbol}{\openbox}
                \renewcommand{\proofname}{\bf Proof:}
               }

 \renewcommand{\proofname}{{\bf Proof:}}

\begin{document}

\title[numerics]
{Numerical analysis of time--dependent  Hamilton--Jacobi equations on networks }

\author{Elisabetta Carlini\and Antonio Siconolfi}
\thanks{ Dipartimento di Matematica, Sapienza, Rome, 00185, Italy
(carlini@mat.uniroma1.it, siconolf@mat.uniroma1.it)}
\maketitle
\begin{abstract}
A new algorithm for time--dependent  Hamilton--Jacobi equations on networks, based on semi--Lagrangian scheme,  is proposed. It  is based on the   definition of viscosity solution for this kind of problems recently given in \cite{Siconolfi}. A thorough convergence analysis, not requiring weak semilimits, is provided. In particular, the check of the supersolution property at the vertices is performed through a dynamical technique which seems new.
The  scheme is efficient,  explicit, allows long time steps, and is suitable to be implemented in a parallel algorithm.
 We present some numerical tests, showing the advantage  in terms of computational cost over the  one proposed in \cite{carlini20}.

\end{abstract}
{\small
\noindent {\bf AMS-Subject Classification:}  
\section{Introductions}
In the last decade there has been a growing interest in the mathematical community for the study of Hamilton--Jacobi equation on networks, see \cite{ACCT13,CMS13,ImbertMonneauZidani2013,Schieborn13}. Motivations run the gamut: from the application to several models, notably  of traffic systems, smart grids and management of computer clusters, to the analysis of some theoretical issues arising from the non--traditional environment where the differential problems are posed.

In this respect, the central problem is to give the right definition of (viscosity) (sub/super) solutions in order to get comparison principles, even assuming the Hamiltonians on arcs of the network with different support to be unrelated. The difficulty is specifically to prescribe appropriate conditions at the vertices, where discontinuities occur.
From this viewpoint, it is clear why time--dependent equations are more demanding than the stationary ones,  for the former, in fact,  the discontinuity  interfaces are one--dimensional, of the form
\[\{(x,t), t \in [0, T]\} \qquad \hbox{with $x$ vertex}\]
while they become $0$--dimensional for the latter, in absence of time.

Understandably, the first complete contribution on the problem has come relatively late in the development of Hamilton--Jacobi problem on networks and junctions, see \cite{ImbertMonneau17}.  In this paper the authors gives the notion of flux limited  viscosity solution and overcome the difficulties at the vertices using special test functions acting simultaneously on  all the local Hamilton-Jacobi equation posed on the arcs adjacent a given vertex.

In \cite{Siconolfi} flux limiters are still used, but local Hamilton--Jacobi equations on any single arcs are tested separately, so that there is no need of special test functions. In addition, networks with a rather general geometry are allowed, the unique restriction being  the nonexistence of self--loops, namely arcs starting and ending at the same vertex.

This approach seems potentially beneficial for numerical approximation. The aim of this paper is to check  this potentiality. We succeed in constructing an algorithm with a simple structure and advantageous in terms of efficiency, with significant  gain in computational time  with respect to \cite{carlini20}. See comparison tests in Section  \ref{sec:test1}.

Literature is not crowded of papers regarding numerics for Hamilton--Jacobi  on networks. A first numerical contribution  for Eikonal equations  can be  found  in \cite{camilli2013approximation}.  Finite difference based schemes  for time--dependent  equations   have been proposed in \cite{costeseque2015convergent} and  \cite{CCM18}. In \cite{costeseque2015convergent}  a convergence analysis is provided,  and in \cite{GK19}  an error estimate is proved  for the scheme proposed in \cite{costeseque2015convergent}. In \cite{carlini20}, a semi--Lagrangian (SL, for short) scheme has been developed, along with a  convergence analysis and an error estimate. It  generalizes
the one  introduced in \cite{camilli2013approximation}, and enables discrete characteristics to cross the junctions, which  makes it unconditionally stable, and allows for large time steps.

Our scheme is based on a semi--Lagrangian approximation, and, besides being  efficient, is  explicit and allows large time steps.  In other terms  it does not requires the CFL condition, which is necessary for explicit finite differences schemes to be stable.  Furthermore, it is suitable to be implemented in a parallel algorithm, even if this possibility  has not been explored here.

The main idea is to work as much as possible with the local equations defined on each arc of the network, which are--one dimensional and  with continuous Hamiltonians.  Since this approach yields  multiple values at the vertices, one for each arc incident on it, a single value is selected through a simple  minimization procedure which takes also into account the flux limiter, see Section \ref{lagrange}. This is actually the unique point where the algorithm feels the network structure.

As  sometimes happens, the simplicity of the algorithm is somehow counterbalanced by some complexity in the convergence analysis. In our case, this specifically occurs in checking  the supersolution property for the limit function at the vertices. As a matter of fact, the definition of supersolution at the vertices  is a crucial point  where the geometry of the network enters  into play.

The check, see Section \ref{supersuper},   is performed through a dynamical technique exploiting in turn the fact, see Lemma \ref{precrux}, that the action functional for curves contained in an arc, starting  and ending at the same vertex, must be bounded from below by a term containing the flux limiter. Arguing by contradiction, it is shown that this estimate is violated constructing with a backward procedure a suitable discrete trajectory and then making it continuous via linear interpolation.

A final remark is that, thanks to an equiLipshitz property of approximated solutions, see Proposition \ref{retire},  we do not use weak semilimits.
\medskip

\subsection{Organization of the paper}
In section 2, we introduce the problem, and give the notion of solution.  Thereafter, we show that the Hamiltonian can be modified, without affecting the solution, in order the corresponding Lagrangian to be  infinite outside a given compact set, which is effective for the minimization problems in the semi--Lagrangian approximation.

Section 3 is devoted to a rough overview of the  algorithm, while in Section 4 we introduce the space--time discretizations we will use in our schemes, and an interpolation operator along with some properties.   In Section  5 we present the discrete evolution problems, and in Section 6 the corresponding convergence analysis when the discretization parameters become infinitesimal.

Sections 7, 8  contain the analysis of sub and supersolution properties of the functions attained at the limit in the approximation procedure. In Section 9 it is stated and proved the main result asserting that the limit function is actually unique and solves the Hamilton--Jacobi problem.

Finally, in Section 10 we present the numerical simulations with a graphical  part illustrating  the features of approximating solutions resulting from the application of the scheme.

There are three appendices: in the first one we describe and give the main properties of the modification of the Hamiltonian performed at the beginning of the approximation, in the second one we recall some facts about spaces convergence. The final one is just a quick user's guide for the interpretation of the  pictures in Section 10, which could be not intuitively immediate  for newcomers.

 \subsection{Notations}  The notations we use throughout the paper are quite standard. We just make precise that given two real numbers $a$, $b$, $a \vee b$, $a \wedge b$ stand for the maximum and the minimum, respectively, of them. The same applies for finitely many real numbers. If $\varphi(s,t)$ is a $C^1$  function in $(0,a) \times  (0,T)$, for some $a$, $T$ positive, with $s$, $t$ playing the role of space and space variables, respectively, we denote by $\psi_t$ the time and by $\psi'$ the space derivative. Given a function $f$ defined in an Euclidean space $\R^D$ and $C$ closed subset of $\R^D$, we say that $f$ is of class $C^1$ in $C$ if it is $C^1$ in some open neighborhood of $C$.  If $f$ is Lipschitz continuous in some $B \subset \R^D$, we denote by ${\mathrm {Lip}} f$ the best Lipschitz constant of $f$ in $B$.

\medskip

\subsection{Networks}   A {\it network}  $\G \subset \R^N$ can be understood as a  piecewise regular $1$-- dimensional manifold. It has  the form
 \[ \G= \bigcup_{\gamma \in \EE} \gamma([0,|\ga|]), \]
 where $\EE$ is a finite collection of regular simple curves, called {\it arcs } of the network, with Euclidean length $|\gamma|$, parametrized by arc length in $[0, |\ga|]$. Note that both  $\ga$  and $\gamma^{-1}:\gamma([0,|\ga|])\to [0,|\ga|]$ are Lipschitz continuous with Lipschitz constant $1$.  If $\ga \in \EE$  the reversed arc
 \[\widetilde \ga(t)= \ga(|\ga|-t)\]
 belongs to $\EE$ as well.
The key condition  on the family $\EE$ is  that arcs with different support can intersect only at the initial or final points, called {\it vertices}. The set of vertices contains finitely many elements and is denoted by $\VV$. We further denote by $\oo(\ga)$ and $\tt(\ga)$, respectively, the initial and final vertex of an arc $\ga$. It is apparent that
\[\oo(\wtd \ga)= \tt(\ga) \qquad\hbox{and} \qquad \tt(\wtd \ga)= \oo(\ga).\]
An arc is   said to be {\it incident} on the vertices it links. An {\it orientation} of $\Gamma$ is a subset $ \EE^+ \subset \EE$ containing  exactly one arc in each pair $\{\ga, \widetilde \ga\}$.  We set
\[\EE^+_x= \{\ga \in \EE^+ \mid \ga \;\hbox{incident on $x$}\} \qquad\hbox{for any $x \in \Gamma$.}\]
Note that $\EE^+_x$ is a singleton if $x \in \Gamma \setminus \VV$, while it contains  in general more than one element   if $x$ is instead a vertex.

The network $\G$ inherits by the ambient space $\R^N$ a metric structure. We assume $\Gamma$ to be {\it connected}, in the sense that  any two vertices are linked by a finite sequence $\{\ga_1, \cdots, \ga_m\}$ of concatenated arcs, namely arcs satisfying
\[\tt  (\ga_j)=\oo (\ga_{j+1}) \qquad\hbox{for any $j= 1, \cdots, M-1$.}\]
The geometry of the networks we consider  is general, the unique limitation being that no {\it self--loops} are admitted, namely arcs $\ga$ with $\oo(\ga)=\tt(\ga)$. This is due to the fact that so far the theory of time--dependent Hamilton--Jacobi equation on networks, that we use, has been developed in absence of self--loops. We do believe that an extension to general networks is possible, but not still available.\\

\section{Basic definitions and assumptions}\label{assu}

\subsection{Hamiltonians on $\Gamma$} We consider an  Hamiltonian on $\Gamma$, namely   a family of Hamiltonians $\wtd H_\ga: [0,|\ga|] \times \R \to \R$ indexed by arcs, which are totally unrelated on arcs with different support, and satisfy   the compatibility condition
\[\wtd H_{\widetilde \ga}(s,\mu)= H_\ga(|\ga|-s,-\mu) \qquad\hbox{for any $\ga$.}\]
We assume  the $\wtd H_\ga$'s to be, for any $\ga$
\begin{itemize}
    \item[{\bf(H1)}]  continuous in both arguments;
    \item[{\bf(H2)}]  convex   in the momentum variable;
    \item[{\bf(H3)}]  superlinear  in the momentum variable, uniformly  in $s$;
    \item[{\bf(H4)}] with  $s \mapsto \wtd H_\ga(s,\mu)$  Lipschitz  continuous in $[0,|\ga|]$ for any $\mu \in \R$.
        \end{itemize}

\smallskip

Taking into account that $\mu \mapsto \wtd H_\ga(s,\mu)$ is locally Lipschitz continuous for any $s$ because of {\bf (H2)}, we deduce from {\bf (H4)} that $\wtd H_\ga$ is locally Lipschitz continuous in $[0,|\ga|] \times \R$, see, for instance, the corollary of Proposition 2.2.6 in \cite{Clarke75}.
Thanks to {\bf (H3)}, a finite Lagrangian $\wtd L_\ga$ can be defined for any $\ga$ through the Fenchel duality formula
\[\wtd L_\ga(s,\la)= \max_{\mu \in \R} \mu \, \la - \wtd H_\ga(s,\mu) \qquad\hbox{for any $s \in [0,|\ga|]$, $\la \in \R$.}\]
Note that {\bf (H1)}--{\bf (H4)} hold with $\wtd L_\ga$ in place of $\wtd H_\ga$, for any $\ga$.
\smallskip
We set
\[\wtd c_\ga= - \max_s \, \min_\mu  \wtd H_\ga(s,\mu) \qquad\hbox{for any arc $\ga$,}\]
and define {\em flux limiter } for the family $\wtd H_\ga$ any function $x \mapsto c_x$ from $\VV$ to $\R$ satisfying
\[c_x  \leq \min_{\ga\in \EE^+_x}\wtd  c_\ga \qquad\hbox{for $x \in \VV$.}\]

We proceed recalling some viscosity solution theory terminology. Given $T >0$ and a continuous function $u: (0,|\ga|) \times (0,T) \to \R$, we call {\it supertangents} (resp.  {\it subtangents}) to $u$ at $(s_0,t_0)$  the viscosity test functions from above (resp. below).

Given a closed subset $C \subset [0,|\ga|] \times [0,T]$, we say that a supertangent  (resp. subtangent) $\varphi$ to $u$ at $(s_0,t_0) \in \partial C$ is {\it constrained to $C$} if $(s_0,t_0)$ is maximizer (resp. a minimizer) of $u - \varphi$ in a neighborhood of $(s_0,t_0)$ intersected with $C$.\\

Given a positive time $T >0$, we consider for any $\ga$ the time--dependent equation
\begin{equation}\label{HJg}  \tag{HJ$_\ga$}
    u_t + \wtd H_\ga(s,u')=0 \qquad\hbox{in $ (0,|\ga|) \times (0,T)$,}
\end{equation}
 a Lipschitz continuous  initial datum $g: \Gamma \to \R$ at $t=0$ plus a flux limiter $c_x$. We call (HJ) the  problem in its totality.
By {\it solution} of (HJ)  we mean a continuous function $ u: \Gamma \times [0,T) \to \R$  such that $u \circ \ga$ is a viscosity solution to  \eqref{HJg} in $(0,|\ga|) \times (0,T)$, for any $\ga$,  is equal to $g \circ \ga$ at $t=0$, and satisfies the following conditions at any $x \in \VV$, $t_0 \in (0,T]$ :
\smallskip

\begin{Definition}\label{defsol}  \hfill
\begin{itemize}
\item[{\bf (i)}]  if
     \[ \frac d{dt} \psi (t_0) < c_x\]
     for some $C^1$ subtangent $\psi$ to $ u(x,\cdot)$ at $t_0 >0$, then
there is  an arc $\ga \in \EE^+_x$  incident on $x$
such that  all the $C^1$  subtangents   $\varphi$,
constrained  to $[0,|\ga|] \times [0,T]$, to
$v \circ \ga$  at  $(\ga^{-1}(x),t_0)$   satisfy
     \[   \varphi_t(\ga^{-1}(x),t_0) +  \wtd H_\ga(\ga^{-1}(x), \varphi'(\ga^{-1}(x),t_0))  \geq 0.\]
\item[{\bf (ii)}]  all $C^1$ supertangents $\psi(t)$ to $u(x,\cdot)$ at $t_0 >0$ satisfy
     \[  \frac d{dt} \psi(t_0) \leq c_x.\]
\end{itemize}
\end{Definition}
Note  that  the arcs $\ga$,  with $\ga \in \EE^+_x$, where   condition {\bf (i)} holds true,  change in function of the time.

We further say that a continuous function $ u$ is {\it subsolution} (resp. {\it supersolution}) of (HJ) if $  u \circ \ga(\cdot,0) \leq g \circ \ga$ (resp. $  u \circ \ga (\cdot,0) \geq g \circ \ga$),   $ u \circ \ga$ is subsolution(resp. supersolution) of \eqref{HJg} for any $\ga$ and condition {\bf (ii)} (resp. {\bf (i)}) holds true at any vertex $x$ for $t_0 >0$.

According to the results of \cite{Siconolfi}, \cite{PozzaSiconolfi}, we have

\begin{Theorem}\label{exun}  There exists one and only one solution of (HJ), and it is in addition Lipschitz continuous in $\G \times [0,T]$.
\end{Theorem}

The numerical approximation of such a solution  is the main aim of the paper.

\medskip

\subsection{Modified problem}

Actually, we will not directly perform the numerical approximation of (HJ), but instead of a problem   with modified Hamiltonians $H_\ga$'s obtained from the $\wtd H_\ga$'s through the procedure described in Appendix \ref{penalized}, for a suitable preliminary choice of a compact interval $I$ of the momentum variable $\mu$ with
\begin{equation}\label{mod1}
  H_\ga(s,\mu) = \wtd H_\ga(s,\mu) \qquad\hbox{for any $\ga$, $s \in [0,|\ga|]$, $\mu \in I$.}
\end{equation}
The rationale behind this change is   that the solution of the modified problem coincide with that of the original one for a suitable choice of $I$, see the forthcoming Theorem \ref{modteo}. The advantage  is that a positive constant $\be_0$, depending on $I$,  can be found such that  the modified Lagrangians $L_\ga$ are infinite when $\la$ is outside $[-\be_0,\be_0]$ for all $s \in [0,|\ga|]$, $\ga$, and globally Lipschitz continuous in $[-\be_0,\be_0]$. This in turn allows handling smoothly the minimization problems involving $L_\ga$ in the forthcoming approximation schemes.

To summarize, the relevant properties of the $H_\ga$/$L_\ga$'s, see Appendix \ref{penalized}, are:
\begin{itemize}
    \item[{\bf(P1)}]  continuity in both arguments;
    \item[{\bf(P2)}]  coercivity   in the second variable, uniformly  in $s$;
    \item[{\bf(P3)}]  $H_\ga(s,\mu) \leq \wtd H_\ga(s,\mu)$ and  $L_\ga(s,\la) \geq \wtd L_\ga(s,\la)$ for any $\la$, $\mu$, $s$;
    \item[{\bf(P4)}] there exists $\be_0$, depending on $I$, see \eqref{mod1}, with $ L_\ga(s,\la) \equiv + \infty$ in $[0,|\ga|] \times \big ( \R \setminus [-\be_0,\be_0] \big )$, $ L_\ga$ is Lipschitz continuous in $[0,|\ga|] \times [-\be_0,\be_0]$.
        \end{itemize}
 We consider  the problem
\begin{equation}\label{HJgm}  \tag{HJ$_\ga$mod}
u_t +  H_\ga(s,u')=0 \qquad\hbox{in $ (0,|\ga|) \times (0,T)$,}
\end{equation}
in place of \eqref{HJg}, plus the same initial datum $g$ and the flux limiter $c_x$ appearing in (HJ). The definition of (sub/super) solution is given as we did for (HJ), see \eqref{defsol}.  As already announced, we are going to prove  that the solution  of (HJmod) coincides with that of (HJ) for an appropriate choice of the interval $I$ satisfying \eqref{mod1}. Note that by {\bf(P3)}
\[c_\ga = - \max_s \, \min_\mu  H_\ga(s,\mu) \geq - \max_s \, \min_\mu  \wtd H_\ga(s,\mu) = \wtd c_\ga\]
for any arc $\ga$, and consequently
\[c_x \leq c_\ga.\]
This shows that
\begin{Lemma} The flux limiter $c_x$ appearing in (HJ) is admissible for the $H_\ga$'s.
\end{Lemma}
We further have, see \cite{Siconolfi}

\begin{Theorem}\label{exunmod}  There exists one and only one solution $u^0$  of (HJmod), and it is in addition Lipschitz continuous in $[0,|\ga|] \times [0,T]$.
\end{Theorem}

To provide an  estimate of the Lipschitz constants of $u^0 \circ \ga$, see \cite{Siconolfi},  we define
\[M_0= \max_\ga \, \min \{m \mid  H_\ga(s,(g\circ \ga)') \leq m \;\hbox{a.e.}\; \forall \, s \in [0,|\ga|]\}.\]
In other terms $M_0$ is the smallest constant for which $g\circ \ga$ is subsolution to $ H_\ga= M_0$ for any $\ga$. We have, see \cite{Siconolfi}
\begin{Proposition}\label{modpro} For any $s \in [0,|\ga|]$, any arc $\ga$
\[ {\mathrm {Lip}}(u^0 \circ \ga)(s,\cdot) \leq  \max \{|\mu| \mid H_\ga(s,\mu) \leq \big (M_0 \vee \max_{x \in \VV} |c_x| \big ) \quad  s \in [0,|\ga|] \}\]
\end{Proposition}

We require   the compact interval $I$  in \eqref{mod1}
to satisfy
\begin{equation}\label{mod2}
 \mu_0 \not \in I \, \Rightarrow \, \left \{\begin{array}{cc}
                                   |\mu_0| > &  \max_\ga  {\mathrm {Lip}}(g\circ \ga) \\
                                    \wtd H_\ga(s,\mu) >&  \big ( (\wtd M_0 +1) \vee \max_{x \in \VV} |c_x| \big ) \, \forall \, |\mu| \geq |\mu_0|, \,\ga,\,  s \in [0,|\ga|], \end{array} \right .
   \end{equation}
    where
\[\wtd M_0= \max_\ga \, \min \{m \mid \wtd H_\ga(s,(g\circ \ga)') \leq m \;\hbox{a.e.}\; s \in [0,|\ga|]\}.\]
It is apparent that
\begin{equation}\label{mod20}
  \mu_0 \in I \, \Rightarrow \, \mu \in I \qquad\hbox{for all $\mu$ with $|\mu| \leq |\mu_0|$,}
\end{equation}
and
\begin{equation}\label{mod21}
 \big ( (\wtd M_0 +1) \vee \max_{x \in \VV} |c_x| \big ) > M_0 \geq \min_\mu H_\ga (s,\mu) \qquad\hbox{for any $\ga$, $s \in [0,|\ga|]$.}
\end{equation}
\smallskip

\begin{Lemma}\label{lemmod}
If $I$ satisfies \eqref{mod2} then the constants $M_0$ and $\wtd M_0$ coincide.
\end{Lemma}
\begin{proof} Assume  $g \circ \ga$ to satisfy, for a given arc $\ga$ and constant $m$
\[ H_\ga(s, (g \circ \ga)') \leq m \qquad\hbox{for a.e. $s$.}\]
Since
\[ |(g \circ \ga)'(s)| \leq {\mathrm {Lip}} (g\circ \ga) \qquad\hbox{for a.e. $s$}\]
we derive that $(g \circ \ga)'(s) \in I$ and consequently
\[ H_\ga(s, (g \circ \ga)'(s)) = \wtd H_\ga(s, (g \circ \ga)'(s))  \qquad\hbox{for a.e. $s$.}\]
This in turn implies that $g \circ \ga$ is also subsolution to $\wtd H_\ga=m$, and consequently
\[\min \{m \mid \wtd H_\ga(s,(g\circ \ga)') \leq m \;\hbox{a.e.}\; s \in [0,|\ga|]\} \leq \min \{m \mid  H_\ga(s,(g\circ \ga)') \leq m \;\hbox{a.e.}\; s \in [0,|\ga|]\}.\]
The opposite inequality comes from {\bf (P3)}. This concludes the argument.
\end{proof}

\smallskip

\begin{Theorem}\label{modteo} If $I$ satisfies \eqref{mod2} then the solutions of (HJ) and (HJmod) coincide.
\end{Theorem}

\smallskip

We preliminarily prove an elementary lemma.

\begin{Lemma}\label{lemmodteo} Let $\wtd f$, $f$ be two convex coercive functions defined on $\R$ with $\wtd f \geq f$. Let $a$ be a constant satisfying $a > \min_\R f$  and
\begin{equation}\label{lemmodteo1}
  f(\mu)= \wtd f(\mu) \qquad\hbox{for any $ \mu$ with $\wtd f(\mu) \leq a$}
\end{equation}
then
\[\wtd J:=\{ \mu \mid \wtd f(\mu) \leq a\} = \{ \mu \mid f(\mu) \leq a\} :=J.\]
\end{Lemma}
\begin{proof} By the assumptions on $\wtd f$,  $\wtd J=[\mu_1,\mu_2]$ for suitable real numbers $\mu_1$, $\mu_2$ with $\wtd f(\mu_1)= \wtd f(\mu_2) =a$, therefore $ f(\mu_1)=  f(\mu_2) =a$ thanks to \eqref{lemmodteo1}, and  $f > a $ in $\R \setminus [\mu_1,\mu_2]$ because $a > \min_\R f$. This implies
\[J = [\mu_1,\mu_2] = \wtd J,\]
as was claimed.
\end{proof}

\smallskip

\begin{proof}[{\bf Proof  of Theorem \ref{modteo}:}] Let us denote by $u^0$ the solution to (HJmod). For any  arc $\ga$, $(s_0,t_0) \in (0,|\ga|) \times (0,T)$,  any $C^1$ viscosity test function $\varphi$ (super or subtangent) to $u^0 \circ \ga$ at $(s_0,t_0)$, we have taking into account Proposition \ref{modpro}
\begin{eqnarray*}
  |\varphi'(s_0,t_0)| &\leq& {\mathrm{Lip}} (u^0 \circ \ga) \\
  &\leq&  \max \{|\mu| \mid H_{\ga}(s,\mu) \leq \big (M_0 \vee \max_{x \in \VV} |c_x| \big ) \quad s \in [0,|\ga|]\}.
\end{eqnarray*}
This implies that there exists $\mu^*$, $s^*$  with
\[|\mu^*| \geq |\varphi'(s_0,t_0)| \qquad\hbox{and} \qquad H_{\ga}(s^*,\mu^*) \leq  \big (M_0 \vee \max_{x \in \VV} |c_x| \big ).\]
We now apply Lemma \ref{lemmodteo} with $f= H_{\ga}(s^*,\cdot)$, $\wtd f= \wtd H_{\ga}(s^*,\cdot)$, $a= \big ( (\wtd M_0 +1) \vee \max_{x \in \VV} |c_x| \big )$, see \eqref{mod21}. As a matter of facts by \eqref{mod2}
\[ \wtd H_{\ga}(s^*,\mu) \leq \big ( (\wtd M_0 +1) \vee \max_{x \in \VV} |c_x| \big ) \,\Rightarrow \, \mu \in I \, \Rightarrow \, \wtd H_{\ga}(s^*,\mu) =  H_{\ga}(s^*,\mu) .\]
We therefore find that
\[ H_{\ga}(s^*,\mu^*) \leq  \big (M_0 \vee \max_{x \in \VV} |c_x| \big ) \, \Rightarrow \wtd H_{\ga}(s^*,\mu^*) \leq \big ( (\wtd M_0 +1) \vee \max_{x \in \VV} |c_x| \big )\]
and consequently $\mu^* \in I$, which in turn implies by \eqref{mod20} that
$\varphi'(s_0,t_0) \in I$. This yields
\[H_{\ga}(s_0,\varphi'(s_0,t_0))= \wtd H_{\ga}(s_0,\varphi'(s_0,t_0)),\]
so that $u^0 \circ \ga$ is solution to \eqref{HJg} in  $(0,|\ga|) \times (0,T)$. Furthermore, $u_0$ satisfies  the condition {\bf (i)} in Definition \ref{defsol} for (HJ)  beacuse $\wtd H \geq H$ and for free  the condition {\bf (ii)}. We conclude that $u^0$ is the solution to (HJ), which ends the proof.
\end{proof}

\smallskip
\begin{Corollary}\label{corcormod} Assume $I$ to satisfy \eqref{mod2}, then the families of Lipschitz subsolutions to (HJ) and (HJmod) coincide.
\end{Corollary}
\begin{proof} Due to $\wtd H_\ga \geq H_\ga$ for any $\ga$, any  subsolution to (HJ) is also subsolution of (HJmod). The converse implication can be proved for a Lipschitz continuous function arguing as in the firt part of the proof of Theorem \ref{modteo}.
\end{proof}

\smallskip

\begin{Notation}\label{notmod}
We denote throughout the paper by $\be_0$ a positive constant satisfying {\bf (P4)} for any $\ga$, in correspondence with an interval $I$ satisfying \eqref{mod2}. We further denote by $\ell_0$ a Lipschitz constant of $g \circ \ga$ in $[0,|\ga|]$ and of $L_\ga$ in $[0,|\ga|] \times [-\be_0,\be_0]$, for any arc $\ga$, namely such that the inequality
\[|L_\ga(s_1,\al_1)-L_\ga(s_2,\al_2)| \leq \ell_0 \big ( |s_1-s_2| + |\al_1-\al_2| \big )\]
holds true for any $(s_1,\al_1)$, $(s_2,\al_2)$ in $[0,|\ga|] \times [-\be_0,\be_0]$, any $\ga$.
\end{Notation}
\bigskip

\section{The algorithm}

In this section we provide a rough idea of our approximation scheme.
We fix  positive integers  $N_\ga$, $\ga \in \EE^+$, $N_T$, and consider finite decompositions of the parameter interval $[0,|\ga|$],  see next section
\[0=s^\ga_0< s^\ga_1< \cdots, s^\ga_{N_\ga-1}< s^\ga_{N_\ga}=|\ga|\]
and  a finite decomposition of the time interval $[0,T]$
\[0=t_0 < t_1< \cdots < t_{N_T-1}< t_{N_T}=T.\]

\medskip

\noindent {\bf step $0$ :} \; we determine numerically the maximal solution of the  equation \eqref{HJgm} in $(0,|\ga|) \times (0,T)$ plus initial condition at $t=0$ given by
\[(g(\ga(s^\ga_0)), \cdots, g(\ga(s^\ga_{N_\ga})))  \qquad\hbox{ for any $\ga \in \EE^+$}\] and denote by
\[ u^1_\ga(s^\ga_i) \qquad i=1, \cdots, N_\ga ,\;
\]
the approximate solutions so obtained. We get, for any vertex $x$, a finite family of values
\[ u^1_\ga (\ga^{-1}(x),t_1) \qquad\hbox{for $\ga \in \EE^+_x$.}\]
We set
\begin{eqnarray*}
  a &=& \min\{ u^1_\ga  (\ga^{-1}(x),t_1) \mid \ga \in \EE^+_x\} \\
 u(x,t_1) &=& \min\{  g(x)  +c_x  \,t_1, \, a\}.
\end{eqnarray*}
 We have therefore determined, for any arc $\ga \in \EE^+$,   a vector
\[(u(\oo(\ga),t_1), u^1_\ga(s^\ga_1,t_1),  \cdots, u^1_\ga(s^\ga_{N_\ga-1},t_1), u(\tt(\ga),t_1))\]
to use as initial value in the next step.

\medskip

\noindent {\bf step $n < N_T$ :} \;  we approximate the maximal solution of the equation \eqref{HJg} in $(0,|\ga|) \times (t_n,T)$ with initial condition at $t= t_n$ given by
\[(u(\oo(\ga),t_n), u^{n}_\ga(s^\ga_1,t_n),  \cdots, u^{n}_\ga(s^\ga_{N_\ga-1},t_n), u(\tt(\ga),t_n))\]
for any $\ga \in \EE^+$,  and denote by
\[ u^{n+1}_\ga(s_i,t_j) \qquad i=1, \cdots, N_\ga, \; j=n+1, \cdots, N_T\]
 the the solutions so obtained. We adjust the value at any vertex $x$ setting
\begin{eqnarray*}
  a &=& \min\{ u^{n+1}_\ga  (\ga^{-1}(x),t_{n+1}) \mid \ga \in \EE^+_x\} \\
 u(x,t_{n+1}) &=& \min\{  u(x,t_n)  +c_x  \,(t_{n+1} - t_n), \, a\},
\end{eqnarray*}
 and determine the vector of values at $t=t_{n+1}$ given by
\[(u(\oo(\ga),t_{n+1}), u^{n+1}_\ga(s^\ga_1,t_{n+1}),  \cdots, u^{n+1}_\ga(s^\ga_{N_\ga-1},t_{n+1}), u(\tt(\ga),t_{n+1})).\]

\bigskip
\section{ Discretization}
We describe in this section  the space--time  discretization and the discrete operators  on which  our  Semi-Lagrangian schemes, to implement the algorithm outlined in Section 2, are based.

\medskip

\subsection{Space--time discretization}
We start by fixing a spatial and a time step denoted by $\Delta x$, $\Delta t$, respectively, that we call {\em admissible pair} if
\begin{equation}\label{admi}
\ 0< \Delta x < |\ga| \quad\hbox{for any $\ga \in \EE^+$}, \quad 0 < \Delta t < T, \quad \Delta x \leq \Delta t,
\end{equation}
and we set
\[\Delta=(\Delta x, \Delta t).\]
We  further define
\[N^\Delta_{\gamma}=\left \lceil{\frac{ |\ga|}{ \Delta x}}\right \rceil > 0 \quad\hbox {for any $\gamma \in \EE^+$, and} \quad N^\Delta_T= \left \lceil \frac T{\Delta t} \right \rceil > 0,\]
where $ \lceil\cdot \rceil$ stands for the ceiling function, namely the least integer greater than or equal to the function argument.

\smallskip

We thereafter set
\begin{eqnarray*}
  s^{\Delta, \ga}_i &=& \left \{ \begin{array}{cc}
        \frac {i |\ga| }{N^\Delta_\ga} &  \quad\hbox{ for $i=0,\dots,N^\Delta_\ga -1$ }\\
               |\ga|& \quad\hbox{ for $ i= N^\Delta_\ga$}
               \end{array} \right . \\
  t^\Delta_n &=& \left \{ \begin{array}{cc}
               \frac {n T}{N^\Delta_T} &  \quad\hbox{ for $n=0,\dots,N^\Delta_T -1$ }\\
                T & \quad\hbox{ for $ n= N^\Delta_T$}
               \end{array} \right .
\end{eqnarray*}
for any $\ga$.  We have therefore associated     to $\Delta$  and any arc $\ga$ the partition of the parameter interval $[0,|\ga|]$ in $N^\Delta_{\gamma}$ subintervals, all of them, except possibly the last, with  length $\Delta x$, and similarly a partition of the time interval $[0,T]$ in $N^\Delta_T$ subintervals  all of them, except possibly the last, with  length $\Delta t$. The last subinterval of the above partitions has length  less than or equal to  $\Delta x$,  $\Delta t$, respectively.

We proceed introducing uniform space--time grids of  $[0,|\ga|] \times [0,T]$, associated to any arc. We  set
\begin{eqnarray*}
  \mathcal S_{\Delta,\ga} &=&  \left \{s^{\Delta,\ga}_i \mid i=0,\dots ,N^\Delta_{\gamma} \right \} \\
 \mathcal T_\Delta &=& \{t^\Delta_n \mid n =0,\dots , N^\Delta_T  \}.
\end{eqnarray*}
and
\begin{eqnarray*}
 \Gamma^T_{\Delta ,\gamma} &=& \SC_{\Delta,\ga} \times \mathcal T_\Delta \\
 \Gamma^T_{\Delta } &=& \underset{\gamma \in \EE^+}\bigcup \ga(\SC_{\Delta,\ga}) \times \mathcal T_\Delta
\end{eqnarray*}

To ease notation, we omit from now on the index $\Delta$ from the above formulae. Accordingly,  in the case where we deal with a sequence of admissible pairs $\Delta_m$, we will write just $m$ instead than $\Delta_m$.

\smallskip

\begin{Proposition}\label{mignotta} Assume the sequence of pairs $\Delta_m$ to be admissible in the sense of  \eqref{admi}, with $\Delta_m $ infinitesimal, then
\begin{itemize}
  \item[{\bf (i)}] $\G^T_{m,\ga} \to [0,|\ga|] \times [0,T]$ for any $\ga \in \EE^+$;
  \item[{\bf (ii)}]  $\G^T_m \to \G \times [0,T]$.
\end{itemize}
The above convergence  must be understood in the sense of Definition \ref{abspaces}.
\end{Proposition}
\begin{proof}
The diagonal length  of any cell of the grids  $\G^T_{m,\ga}$ and $\G^T_m$ are estimated from above by
\[ \sqrt{ (\Delta_m x)^2 + (\Delta_m t)^2} \quad\hbox{and} \quad \sqrt{ ( \ell^* \Delta_m x )^2 + (\Delta_m t)^2}, \]
respectively, where $\ell^*$ is the maximum of the Lipschitz constants of the arcs $\gamma \in \EE^+$.
Both lengths become infinitesimal as $m$ goes to $+ \infty$.
\end{proof}

\medskip

\subsection{An operator on discrete functions}
We denote by $B(\Gamma^T)$ and $B(\Gamma^T_\ga)$, for $\ga \in \EE^+$, the spaces of functions from $\Gamma^T$,
$\Gamma^T_\gamma$, respectively,  to $\R$.  If $v \in B(\Gamma^T)$ and $\ga \in \EE^+$, we set
\[v \circ \ga(s,t)= v(\ga(s),t)   \qquad\hbox{for $(s,t) \in \Gamma^T_\ga$}.\]
It is apparent that $v \circ \ga \in B(\Gamma^T_\ga)$.

\smallskip

Given an arc $\ga$ and  $w \in B(\Gamma^T_\ga)$, we define the {\it interpolation operator}
\[I_\ga[w](s,t) = w(s_i,t)+\frac{s-s_i}{s_{i+1}-s_i}(w (s_{i+1},t) - w(s_i,t) )\]
with $s_{i}$, $s_{i +1}$ in $\mathcal S_\ga$ and $s \in [s_i, s_{i+1}]$.
We record for later use:

 \begin{Lemma}\label{bastille} Given an arc $\ga$, $w \in B(\G^T_\ga)$, $t \in \mathcal T$, if $w(\cdot,t)$  is  a Lipschitz continuous function from $\mathcal S_\ga$ to $\R$,  then $I_\ga[w](\cdot,t)$  stays   Lipschitz continuous in $[0,|\ga|]$ with the same Lipschitz  constant.
 \end{Lemma}
 \begin{proof} We denote by $\ell_w$ the Lipschitz constant of $w(\cdot,t)$, and consider   $s' > s$ in $[0,|\ga|]$.  We first  assume  that $s$, $s'$ lie  in two different intervals $[s_1,s_2]$, $[s'_1,s'_2]$, respectively,  of the decomposition of $[0,|\ga|]$ associated to $\mathcal S_\ga$, so that
 \begin{equation}\label{bastille1}
  s'-s = (s'-s'_1) + (s'_1-s_2) + (s_2-s).
 \end{equation}
 To ease notation we set
 \begin{equation}\label{bastille2}
  \la= \frac{s-s_1}{\Delta x}, \qquad \la'= \frac{s'-s'_1}{\Delta x},
 \end{equation}
 note that
 \begin{equation}\label{bastille3}
   1- \la= \frac{s_2-s}{\Delta x}, \qquad 1- \la'= \frac{s'_2-s'}{\Delta x}
 \end{equation}
 Taking into account \eqref{bastille1}, \eqref{bastille2}, \eqref{bastille3}, we have
 \begin{eqnarray*}
 &&|I_\ga[w](s,t) - I_\ga[w](s',t)| = | (1-\la) \, w(s_1,t) + \la \, w(s_2,t) - (1-\la') \, w(s'_1,t) - \la' \, w(s'_2,t)| \\
  &\leq&  | (1-\la) \, w(s_1,t) + \la \, w(s_2,t) - w(s_2,t)| + |w(s'_1,t) - w(s_2,t)|\\ &+& |w(s'_1,t) - (1-\la') \, w(s'_1,t) - \la' \, w(s'_2,t)|\\
  &\leq& (1-\la)|w(s_1,t)-w(s_2,t)| + \ell_w \, (s'_1-s_2) + \la' |w(s'_2)- w(s'_1)|\\ &\leq&  \frac{s_2-s}{\Delta x}\, \ell_w \, \Delta x +\ell_w \, (s'_1-s_2) +
  \frac{s'-s_1'}{\Delta x} \, \ell_w \, \Delta x\\
  &=& \ell_w \, [(s_2-s) + (s'_1-s_2) + (s'-s'_1)] = \ell_w \, |s-s'|.
 \end{eqnarray*}
 It is left the case where $s$ and $s'$ belong to the same interval $[s_1,s_2]$ of the decomposition associated to $\mathcal S_\ga$. We are still assuming $s' >s$, we keep the same definition of $\la$  as in \eqref{bastille2} while $\la'$ now is set as $\frac{s'-s_1}{\Delta x}$. We have
 \begin{eqnarray*}
 |I_\ga[w](s,t)- I_\ga[w](s',t)| &=& | (1-\la) \, w(s_1,t) + \la \, w(s_2,t) - (1-\la') \, w(s_1,t) - \la' \, w(s_2,t)| \\
 &=& (\la' -\la)|w(s_1,t) - w(s_2,t)| \leq \frac {s'-s}{\Delta x} \, \ell_w \, \Delta x= \ell_w \, (s'-s).
 \end{eqnarray*}
 \end{proof}

\bigskip

\section{Evolutive problems}

\subsection{Semi Lagrangian schemes}\label{lagrange}
We define  on each arc $\gamma \in \EE^+$ the  operator
  \[ S_\ga: B(\Gamma^T_\ga) \to B(\Gamma^T_\ga)\]

\[S_\gamma[ v](s,t)=\min_{\frac{s- |\ga|}{\Delta t}\leq \alpha\leq \frac{s}{\Delta t}} \{I_\ga[v]( s-\Delta t \alpha,t )+\Delta t L_{\ga}(s,\alpha) \}
\]
 for  $(s,t) \in \Gamma^T_{\ga}$.

\medskip

We consider the case where the operators $S_{m,\ga}$, corresponding to a sequence of infinitesimal pairs $\Delta_m$,  are applied to the restriction in $ \G^T_{m,\ga}$ of a $C^1$ function.

\smallskip
\begin{Proposition}\label{nerobuono} Let $\Delta_m= (\Delta_m x, \Delta_m t)$ be a sequence of admissible pairs which becomes infinitesimal as $m$ goes to infinity.
Then for any arc $\ga \in \EE^+$ and for any function $\psi: (0,|\ga|) \times (0,T) \to \R$ of class $C^1$, we have
\begin{eqnarray*}
  \frac{S_{m,\gamma}[\psi](\cdot, \cdot)-\psi(\cdot,\cdot)} {\Delta_m t} &\to&  - H_{\ga}(\cdot,\psi'(\cdot,\cdot)) \\
  \frac{\psi(\cdot, \cdot - \Delta_m t)- \psi(\cdot,\cdot)}{\Delta_m t} &\to &- \psi_t(\cdot,\cdot)
\end{eqnarray*}
uniformly in $(0,|\ga|) \times (0,T) $ as $\G_{m,\ga} \cap \big ((0,|\ga|) \times (0,T) \big )\to \big ( (0,|\ga|) \times (0,T) \big )$. Here {\em  uniformly} must be understood in the sense of Definition \ref{deftoy}.
\end{Proposition}
\begin{proof}  We fix an arc $\ga \in \EE^+$. We consider $(s,t) \in (0,|\ga|) \times (0,T)$ and  $(s_m,t_m)
\in \G_{m,\ga} \cap \big ((0,|\ga|) \times (0,T) \to (0,|\ga|) \times (0,T) \big )$ with
\[(s_m,t_m) \to (s,t).\]
 We obtain  through  a first order Taylor expansion
\begin{eqnarray}
&&S_{m,\gamma}[ \psi](s_m,t_m)-\psi(s_m,t_m) \label{nerobuono2} \\ &=&\min_{\frac{s_m-1}{\Delta_m t}\leq \alpha\leq \frac{s_m}{\Delta_m t}}\left (-\psi(s_m,t_m)+ I_{m,\ga}[\psi]( s_m-\Delta_m t \alpha,t_m )+\Delta_m t L_\gamma(s_m,\alpha)\right) \nonumber\\
&=&\min_{\frac{s_m-|\ga|}{\Delta_m t}\leq \alpha\leq \frac{s_m}{\Delta_m t}}\Big(-\Delta_m t \alpha\,\psi'( s_m,t_m)+\Delta_m t L_\gamma(s_m,\alpha)+\oo(\Delta_m t)\Big). \nonumber
\end{eqnarray}
Since  $s \in (0,|\ga|)$, then
\[\lim_m \frac{s_m-|\ga|}{\Delta_m t}=- \infty \quad\hbox{and} \quad \lim_m \frac{s_m}{\Delta_m t}=+ \infty,\]
 the minimum  in \eqref{nerobuono2} is consequently over the interval $[-\be_0,\be_0]$, for $m$ large enough, thanks to Lemma \ref{toy}. By dividing by $\Delta_m t$, and  sending $m$ to infinity,   we therefore get
\begin{eqnarray*}
 \lim_{m\to \infty}\frac{S_{m,\gamma}[\psi](s_m,t_m)-\psi(s_m,t_m)} {\Delta_m t} &=&   \lim_{m\to \infty} \left [- \max_{|\alpha| \leq \be_0}\Big( \alpha\, \psi'(s_m,t_m)- L_\gamma(s_m,\alpha)\Big) \right ] \\
   &=&  \lim_{m\to \infty}  - H_{\ga}(s_m,\psi'( s_m,t_m)) = - H_\ga(s,\psi'( s,t)).
   \end{eqnarray*}
This shows the first part of the assertion. The second part can be proved straightforwardly: let $(s,t) \in [0,|\ga|] \times [0,T]$,   $(s_m,t_m)
\in \G_{m,\ga} \cap \big (0,|\ga|) \times (0,T) \big )$ with $(s_m,t_m) \to (s,t)$. We have for $m$ large enough
\[\frac {\psi(s_m,t_m - \Delta_m t) - \psi(s_m, t_m )}{\Delta_m t} = - \psi_t(s_m,\tau_m)\]
for a suitable $\tau_m \in [t_m - \Delta_m t, t_m]$. This implies that $\tau_m \to t$ as $m \to + \infty$ and
\[\frac {\psi(s_m,t_m -\Delta_m t) - \psi(s_m, t_m )}{\Delta_m t} \to - \psi_t(s,t),\]
as was asserted.
\end{proof}

\smallskip

It is easy to check  the following.
\begin{Proposition}\label{propri} For any arc $\ga$ the  operator $S_\ga$ is
\begin{enumerate}
\item[ i)] {\rm{monotone}}, i.e. given  $ w_1, \,w_2 \in B(\Gamma_\ga^T)$ with $ w_1\leq  w_2$, we have
\[ S_\ga[ w_1](s,t) \leq S_\ga[ w_2](x,t) \qquad\hbox{ for all $(s ,t)\in \Gamma_\ga^T$;}
\]
\item[ ii)] {\rm{invariant by addition of constants}}, namely given $w \in B(\Gamma_\ga^T)$, we have
\[ S[ w+C](s,t)=S[w](s,t)+C \qquad\hbox{for any constant $C$, any $(s,t) \in \Gamma_\ga^T$.}\]
\end{enumerate}
\end{Proposition}

\smallskip

We lift the $S_\ga$'s to  $B(\G ^T )$  defining
\[S: B(\G^T) \to B(\G^T) \] via
\begin{equation}\label{zero}
  S[ v](x,t) = \{S_\gamma[v\circ \ga](\ga^{-1}(x),t) \mid \gamma \in \EE^+_x\} \quad\hbox{if $(x,t)  \in \Gamma^T$, $x \not\in  \VV$}
\end{equation}
and through the following two steps procedure if instead $x$ is a vertex
\begin{eqnarray}
 \widetilde S[v](x,t) &=& \min \{S_\gamma[v\circ \ga](\ga^{-1}(x),t) |\, \gamma \in \EE^+_x \} \label{uno} \\
  S[ v](x,t) &=& \min \{ \widetilde S[v](x,t), u(x,t) + c_x\Delta t\} \label{due}
\end{eqnarray}

\medskip

\subsection{Evolutive discrete problems}
Given a Lipschitz continuous datum $g$ in $\G$, we  consider the following evolutive problem with the operator $S$   defined in \eqref{zero}, \eqref{uno}, \eqref{due}
\begin{equation}\label{eq:scheme}
\left \{\begin{array}{cccc}
  v(x,0) & =& g(x) & (x,0) \in \G^T\\
 v(x,t) & =& S[v](x,t-\Delta t)) & (x,t) \in \G^T, \, t >0
\end{array} \right .
\end{equation}
 Given an arc $\ga \in \EE^+$, we use the solution $v$ of the above problem as lateral boundary datum  for the following one
\begin{equation}\label{eq:schemeg}
\left \{\begin{array}{cccc}
  w_\ga(s,0) &=& g\circ \ga(s) & s \in \SC_\ga \\
  w_\ga(0,t) &=& v(\ga(0),t)   & t \in \mathcal T \cap (0,T]\\
  w_\ga(|\ga|,t) &=& v(\ga(|\ga|),t) & t \in \mathcal T \cap (0,T]\\
  w_\ga(s,t) &=& S_\ga[w_\ga](s,t-\Delta t) & (s,t) \in  \G^T_\ga \cap \big ( (0,|\ga|) \times (0,T] \big )
\end{array} \right .
\end{equation}

\smallskip

\begin{Proposition}\label{iran} The function $v \circ \ga$, with $v$   solution to \eqref{eq:scheme}, solves  \eqref{eq:schemeg}.
\end{Proposition}
\begin{proof} The function $v \circ \ga$ trivially satisfies the boundary conditions.  We then  take $(s,t) \in \G^T_\ga \cap \big ( (0,|\ga|) \times (0,T] \big )$, then since $v$ is solution of \eqref{eq:scheme}
\[v \circ \ga(s,t)= v(\ga(s),t)= S[v](\ga(s),t-\Delta t).\]
Taking into account that $\ga(s)$  is not a vertex we further have by \eqref{zero}
\[S[v](\ga(s),t) = S_\gamma[v\circ \ga](s,t) .\]
Combining the two previous formulae, we finally get
\[v \circ \ga(s,t) = S_\gamma[v\circ \ga](s,t- \Delta t).\]
This concludes the proof.

\end{proof}

\bigskip

\section{Asymptotic analysis}

We denote by  $\Delta_m$  a   sequence of admissible pairs with $\Delta_m \to 0$  and by $v_m$ the  solution to \eqref{eq:scheme}  with $\Delta_m$ in place of $\Delta$. We further  denote by $I_{m,\ga}$  the interpolation operator related to $\Delta_m$.

The aim of this section is to show that the $v_m$'s uniformly converge as $m \to + \infty$, up to a subsequence, to a Lipschitz continuous function $u$ defined in $\G \times [0,T]$.

We arbitrarily fix  an arc $\ga$,  we start by pointing out a relevant equiLipschitz property for the sequence $v_m \circ \ga$.

\smallskip

\begin{Proposition}\label{retire}  The functions  $(s,t) \mapsto v_m \circ \ga$ are Lipschitz continuous in $ \G^T_{m,\ga}$
with Lipschitz constant independent of $m$.
\end{Proposition}
\begin{proof}
We  start by showing that   $s \mapsto v_m \circ \ga(s,t)$ is Lipschitz  continuous in $\SC_{m,\ga}$  for any $m$, $t \in \mathcal T_m$, $m \in \N$, with Lipschitz constant   independent of $t$ and $m$.  More precisely, we claim that a Lipschitz constant for $v_m(\cdot, t^m_k)$  is given by
\begin{equation}\label{induce}
 (1+ k \, \Delta_m t) \, \ell_0 \leq (1 + T) \, \ell_0 \qquad\hbox{for $k=0, \cdots N^{m}_T$,}
\end{equation}
where  $\ell_0$ is a Lipschitz constant of  both $g \circ \ga$   in $[0,|\ga|]$  and  $L_\ga$ in $[0,|\ga|] \times [-\be_0, \be_0]$, see Notation \ref{notmod}. We show the claim by induction on $k$.

Since it is trivial for $k=0$, we  go on assuming $v_m(\cdot,t^m_k)$ to be Lipschitz continuous in $\mathcal S_{m,\ga}$ with Lipschitz constant as in \eqref{induce}, and prove that $v_m(\cdot,t^m_{k+1})$  keeps the  Lipschitz continuity, and  the Lipschitz constant is given replacing in \eqref{induce}  $k$ by $k+1$. Let  $s_1$, $s_2$ be in $\SC_{m,\ga}$,
we denote  by  $\alpha_2$  an optimal element   for $v_m \circ\ga(s_2,t^m_{k +1})$, and assume first that $s_1-\Delta_m t \alpha_2 \in [0,|\ga|]$,  then
 we can write taking into account Lemma \ref{bastille}
\begin{eqnarray*}
 && v_m \circ \ga (s_1,t^{m}_{k +1}) - v_m \circ \ga (s_2,t^{m}_{k +1}) \\&\leq& I_{m,\ga}[v_m \circ \ga]( s_1-\Delta_m t \alpha_2,t^{m}_k) +\Delta_m t L_\ga(s_1,\alpha_2) \\ &-&I_{m,\ga}[v_m \circ \ga](s_2-\Delta_m t \alpha_2,t^{m}_k) - \Delta_m t L_\ga(s_2,\alpha_2)
  \\ &\leq& (1 + k \, \Delta_m t )\, \ell_0\, |s_1-s_2 | + \Delta_m t \, \ell_0 \, |s_1-s_2|\\&=&  ( 1 +(k +1)\, \Delta_m t) \, \ell_0 \, |s_1-s_2|,
 \end{eqnarray*}
as was claimed.  We proceed assuming  that $s_1 - \al_ 2 \Delta_m t >|\ga| $. We deduce that $\al_2 < 0$,  $s_1 > s_2$, we define
\[\al_1= \frac{s_1 -|\ga|}{\Delta_m t} > \al_2 > - \be_0.\]
We can write
\begin{eqnarray*}
 && v_m \circ \ga (s_1,t^{m}_{k +1}) - v_m \circ \ga (s_2,t^{m}_{k +1}) \\&\leq&   I_{m,\ga}[v_m \circ \ga]( |\ga|,t^{m}_k) +\Delta_m t L_{\ga}(s_1,(s_1-|\ga|)/\Delta_m t) \\ &-&I_{m,\ga}[v_m \circ \ga](s_2-\Delta_m t \alpha_2,t^{m}_k) - \Delta_m t L_{\ga}(s_2,\alpha_2) \\ &\leq& (1  + k \, \Delta_m t )\, \ell_0 \, (|\ga|- s_2+\Delta_m t \al_2)  + \Delta_m t \, \ell_0 \, (s_1-s_2) + \ell_0 \,  (s_1 -|\ga|   - \al_2 \Delta_m t  )\\ &\leq& (1 + k \, \Delta_m t )\, \ell_0 \, (|\ga|- s_2+\Delta_m t \al_2)  + \Delta_m t \, \ell_0 \, (s_1-s_2) + (1 + k \, \Delta_m t ) \ell_0 \, (s_1 -|\ga|   - \al_2 \Delta_m t)
  \\&=&(1+(k+1) \Delta_m t) \ell_0 \, |s_2-s_1|.
 \end{eqnarray*}
 The case where $s_1 - \al_ 2 \Delta_m t <  0$ can be treated along the same lines. The induction argument is therefore concluded and  \eqref{induce} has been proved. We proceed showing the claimed Lipschitz continuity in time. We consider
\[(s,t) \in  \SC_m \times \big ( \mathcal T_m \cap \{t >0\} \big )\]
and denote by $\alpha_0$ an optimal element
for $S_{m,\ga}[v_m \circ \ga](s,t-\Delta_m t)$. We have
\begin{eqnarray*}
  &&|v_m \circ \ga(s,t - \Delta_m t) - v_m \circ \ga(s,t)  |\\ & =& |I_{m,\ga}[v_m \circ \ga](s,t - \Delta_m t) - I_{m,\ga}[v_m \circ \ga](s-\Delta_m t \alpha_0,t-\Delta_m t) - \Delta_m t \, L_\ga(s,\alpha_0)| \\
   &\leq&  ( 1 + T )\, \ell_0  \, |\Delta_m t \, \alpha_0| + \left (\max_{|\alpha| \leq \be_0 \atop s \in [0,|\ga|]} |L_\ga(s,\alpha)| \right ) \Delta_m t
   \\&\leq& \left [   (1 + T) \, \ell_0 \, \be_0  +
   \left (\max_{ |\alpha| \leq \be_0 \atop s \in [0,|\ga|]} |L_\ga(s,\alpha)| \right ) \right ]\, \Delta_m t.
\end{eqnarray*}
This concludes the proof.
\end{proof}

\smallskip
\begin{Notation}\label{nota1}
Since the arcs are finite, we have, thank of the previous result, that there  exists a positive constant denoted in what follows by $\wtd \ell$ with
\[|v_m\circ \ga(s_1,t_1) - v_m\circ\ga (s_2,t_2)| < \wtd \ell \, (|s_1-s_2| + |t_1-t_2|)\]
for any arc $\ga$, $s_1$, $s_2$ in
$\SC_{m,\ga}$, $t_1$, $t_2$ in $\mathcal T_m$.
\end{Notation}

\smallskip

We smoothly derive from Proposition \ref{retire}

\begin{Proposition}\label{nogymn}
For any arc $\ga$, the functions $v_m \circ \ga$   uniformly converge, up to subsequences,  when
\[\G^T_{m,\ga}  \to [0,|\ga|] \times [0,T]\]
 to a Lipschitz continuous function denoted by $\wtd u_\ga$, defined in $[0,|\ga|] \times [0,T]$, with Lipschitz constant $\wtd \ell$.
\end{Proposition}
\begin{proof} The $v_m \circ \ga$'s are equiLipschitz continuous in $\G^T_{m,\ga}$ and equibounded, because they  take the same initial datum at $t=0$. By linear interpolation of the values of $v_m \circ \ga$ on the grid $\G_{m,\ga}^T$, we obtain a sequence of equiLipschitz  and equibounded functions $\wtd u_{m,\ga}$   defined in $[0,|\ga|] \times [0,T]$. By applying Ascoli theorem we get  that $\wtd u_{m,\ga}$ uniformly converge, up to subsequences,  to    a Lipschitz function denoted by $\wtd u_\ga$. It is clear that the same holds true for the original sequence $v_m \circ \ga$ as $\G^T_{m,\ga}  \to [0,|\ga|] \times [0,T]$. It is easy to check that the Lipschitz constant of $\wtd u_\ga$ is  $\wtd \ell$.
\end{proof}

\medskip

Taking into account that the arcs are finite, we find a common subsequence  for which uniform convergence of $v_m \circ \ga$ takes place for any $\ga$, and we still denote the uniform limit function $\wtd u_\ga$. We can therefore define a Lipschitz function   $u$ in $\G \times [0,T]$ setting
\begin{equation}\label{quasiquasi0}
u(x,t) = \wtd u_\ga(\ga^{-1}(x),t) \qquad\hbox{for $\ga \in \EE^+_x$,}
\end{equation}
which  conversely yields
\begin{equation}\label{quasiquasi}
 \wtd u_\ga = u \circ \ga \qquad\hbox{in $[0,|\ga|] \times [0,T]$, for any arc $\ga$.}
\end{equation}

\begin{Lemma}\label{prequasimain} The above definition of $u$ is well posed and $v$ is Lipschitz continuous.
\end{Lemma}
\begin{proof} For the well posedness of \eqref{quasiquasi0}, we check that if $x$ is a vertex incident on two different arcs $\ga_1$, $\ga_2$ in $\EE^+$, then
\begin{equation}\label{prequasimain1}
 \wtd u_{\ga_1}(\ga_1^{-1}(x),t)= \wtd u_{\ga_2}(\ga_2^{-1}(x),t) \qquad\hbox{for any $t \geq 0$.}
\end{equation}
To ease notation, we assume that whole sequence $v_m \circ \ga$ is convergent for any $\ga$. Taking therefore into account that $\wtd u_{\ga}$ is the  uniform limit of $v_m \circ \ga$, for any $\ga$, we have for  $t \in [0,T]$,
\begin{eqnarray*}
  \wtd u_{\ga_1}(\ga_1^{-1}(x),t) &=& \lim_m v_m\circ\ga_1(\ga_1^{-1}(x),t_m) \\
  \wtd u_{\ga_2}(\ga_2^{-1}(x),t) &=& \lim_m v_m\circ\ga_2(\ga_2^{-1}(x),t_m)
\end{eqnarray*}
where $t_m \in \T_m$ and $t_m \to t$,  and
\begin{eqnarray*}
  v_m\circ\ga_1(\ga_1^{-1}(x),t_m) &=& v_m(\ga_1\circ\ga_1^{-1}(x),t_m) = v_m(x,t_m) \\
  v_m\circ\ga_2(\ga_2^{-1}(x),t_m) &=& v_m(\ga_2\circ\ga_2^{-1}(x),t_m) = v_m(x,t_m)
\end{eqnarray*}
which  shows \eqref{prequasimain1}.
According to Proposition \ref{nogymn}, the function $\wtd u_\ga$ is Lipschitz continuous in $[0,|\ga|] \times [0,T]$ for any arc $\ga$ with Lipschitz constant $\wtd \ell$, taking into  account \eqref{quasiquasi0} and that  the  $\ga^{-1}$ are equiLipschitz continuous, we deduce that $u$ is Lipschitz continuous, as was claimed.
\end{proof}

We finally have:

\begin{Theorem} \label{quasimain}
The  sequence of solutions  $v_m$ to  \eqref{eq:scheme}  uniformly  converges, up to subsequences, to the Lipschitz continuous function  $u$ defined in \eqref{quasiquasi0} as  $\G^T_m \to \G \times [0,T]$.
\end{Theorem}
\begin{proof} We assume again for simplicity that the whole sequence $v_m \circ \ga$ converges to $\wtd u_\ga$ for any $\ga$. We consider an arbitrary point  $(x_0,t_0) \in \G \times [0,T]$, and  a sequence $(x_m,t_m) \in \G^T_m $ converging to it. There exists an arc $\ga$ whose support contains $x_0$ and the $x_m$'s, up to a subsequence.
Since $\ga^{-1}(x_m) \to \ga^{-1}(x_0)$ we derive from the very definition of $u$ and Proposition \ref{nogymn}
\[v_m(x_m,t_m) = v_m \circ \ga(\ga^{-1}(x_m),t_m) \to \wtd u_\ga( \ga^{-1}(x_0),t_0)= u(x_0,t_0).\]
By repeating the argument  for any arc $\ga$ with support containing $x_0$, we show the assertion.
\end{proof}

\bigskip

\section{Solution properties of any  limit function }
We consider a limit function   $u$ as in  the in the statement of Theorem \ref{quasimain}. We first show that $u \circ \ga$ is solution of \eqref{HJgm} in $(0,|\ga|) \times (0,T)$ for any arc $\ga$, as second step we prove that $u$ satisfies the condition {\bf (ii)} of Definition \ref{defsol} of subsolution at any vertex.  To ease notations, we assume  throughout the section that the whole sequence $v_m$ converges to $u$.
\medskip

\subsection{Solution properties on any arc}
We exploit the asymptotic analysis of the previous section  plus the properties of the numerical operator $S_{\Delta,\ga}$ summarized in Proposition \ref{propri}, to get:

\begin{Proposition} \label{main}  The  function $u \circ \ga$ is a  solution to \eqref{HJgm} with initial datum $g \circ \ga$ at $t=0$ for any arc $\ga$ .
\end{Proposition}
\begin{proof} Given an arc $\ga$,  we exploit that the $v_m\circ \ga$'s  uniformly converge to $u \circ \ga$ in
$[0,|\ga|] \times [0,T]$, as established in Proposition \ref{nogymn} and \eqref{quasiquasi}.
Let $\varphi$ be a $C^1$ supertangent, that we can assume strict without loosing generality, to $ u \circ \ga$ at $(s^*,t^*) \in (0,|\ga|) \times (0,T)$. We denote by $U$ a compact  neighborhood of $(s^*,t^*)$ in $(0,|\ga|) \times (0,T)$ such that  $(s^*,t^*)$ is the unique maximizer of $u \circ \ga - \varphi $ in $U$.  Let  $(s_m,t_m)$ be a sequence of maximizers of $v_m  \circ \ga- \varphi$ in  $U \cap \G^T_{m, \ga}$, then
\[ (s_m,t_m) \to (s^*,t^*) \qquad\hbox{and} \qquad v_m(\ga(s_m),t_m) \to u(\ga(s^*),t^*)\]
by Proposition \ref{max}.
We  have
\[\varphi_m (s,t) :=\varphi(s,t)+ [v_m(\ga(s_m),t_m) - \varphi(s_m,t_m)] \geq v_m(\ga(s),t)\]
in $U \cap \Gamma_{m, \ga}$, which  implies by Proposition \ref{propri}
\begin{eqnarray*}
 S_{m,\gamma}[v_m \circ \ga](s,t ) &\leq& S_{m,\gamma}[\varphi_m](s,t) \\
  &=& S_{m,\gamma}[\varphi](s,t) + [v_m(\ga(s_m),t_m) - \varphi(s_m,t_m)]
\end{eqnarray*}
and consequently, exploiting that $S_{m,\gamma} [v_m\circ \ga](s_m,t_m - \Delta_m t )= v(\ga(s_m),t_m )$
\[S_{m,\gamma}[\varphi](s_m,t_m - \Delta_m t)- \varphi(s_m,t_m) \geq 0.\]
We therefore get
\[   \frac {\varphi(s_m,t_m-\Delta_m t  ) - \varphi(s_m,t_m )}{\Delta_m t} +
\frac{S_{m,\ga }[\varphi](s_m,t_m- \Delta_m t)- \varphi(s_m,t_m- \Delta_m t)}{\Delta_m t} \geq 0,  \]
and passing to the limit as $m \to + \infty$, we  find in accordance with Proposition \ref{nerobuono}
\[\varphi_t(s^*,t^*) + H_\ga(s^*,\varphi'(s^*,t^*)) \leq 0.\]
Arguing in the same way, we obtain for any  $(s_0,t_0) \in (0,|\ga|) \times (0,T)$, any $C^1$ subtangent $\phi$ to $v \circ \ga$ at $(s_0,t_0)$ the inequality
\[\phi_t(x_0,t_0) + H_\ga(s_0,\phi'(s_0,t_0)) \geq 0.\]
\end{proof}
\bigskip
\subsection{Subsolution property at the vertices}

\begin{Proposition} \label{bimain} The  function $u$ satisfies the subsolution condition {\bf (ii)} of Definition \ref{defsol} at any  vertex $x$.
\end{Proposition}
\begin{proof} Let $\psi(t)$ be a $C^1$ supertangent, which can be assumed strict, to $u(x,\cdot)$ at a time $t_0 >0$. Then $t_0$ is unique maximizer of $u(x,\cdot) - \psi$ in $[t_0- \de,t_0+\de]$, for a suitable $\de >0$. We consider a sequence of maximizers $t_m$ of $v_m(x,\cdot) - \psi$ in $[t_0- \de,t_0+\de] \cap \mathcal T_m$, and deduce from Proposition \ref{max} that
\[ t_m \to t_0 \qquad\hbox{and}\qquad \lim_m v_m(x,t_m)=u(x,t_0).\]
Taking account \eqref{due}, that $v_m$ is solution to \eqref{eq:scheme}, we have
\[c_x \geq \frac{v_m(x,t_m)-v_m(x,t_m -\Delta_m t)}{\Delta_m t} \geq
  \frac{\psi(t_m)-\psi(t_m -\Delta_m t)}{\Delta_m t}. \]
Passing to the limit as $m \to +\infty$ and exploiting Proposition \ref{nerobuono},  we finally get
\[c_x \geq \ \frac d{dt} \psi(t_0).\]
This concludes the proof.
\end{proof}

\bigskip

\section{Supersolution property at the vertices}\label{supersuper}

The proof  that any limit function $u$ satisfies  the supersolution property for  (HJmod)  at any vertex, see item {\bf (i)} in Definition \ref{defsol} with $H_\ga$ in place of $\wtd H_\ga$, is  more demanding and requires a rather articulated proof. It is also necessary strengthening the assumption on the infinitesimal pairs $\Delta_m$ requiring in addition
\begin{equation}\label{opiccolo}
  \lim_m \frac{\Delta _m x}{\Delta_m t}=0.
\end{equation}
 Throughout the section we assume, to ease notation, that the whole sequence $v_m$ converges to $u$.

\subsection{Optimal trajectories}\label{ops}

 In this section  we put in relation the solution $v$ of \eqref{eq:scheme}, for  a fixed  pair $\Delta=(\Delta x,\Delta t)$, to the action on a  suitable curve supported in $[0,|\ga|]$. Let $x_0$ be  a vertex, $t_0 > 0$ in $\mathcal T$, we   crucially  assume  that
\begin{equation}\label{nuovo1}
 v (x_0,t_0)= S_{\gamma}[ v \circ \ga](\ga^{-1}(x_0),t_0 - \Delta t),
\end{equation}
for a suitable arc $\ga \in \EE^+_{x_0}$,  see the definition in \eqref{uno}, \eqref{due} of the operator $S[\cdot]$ at $(x,t)$ when $x$ is a vertex.

If condition \eqref{nuovo1} holds true, we first construct backward a discrete trajectory  $\xi$ from an interval   of $\mathcal T$, namely a  finite sequence of consecutive times in $\mathcal T$, to $\mathcal S_\ga$, ending at $\ga^{-1}(x_0)$.
We set $s_0=\ga^{-1}(x_0)=\xi(t_0)$,  and   denote by $\al_{s_0}$ an optimal $\al$ for
\[ S_\gamma[ v](s_0,t_0-\Delta t)=\min_{\frac{s_0-|\ga|}{\Delta t}\leq \alpha\leq \frac{s_0}{\Delta t}} \{I_\ga[v]( s_0-\Delta t \alpha,t_0 - \Delta t )+\Delta t L_{\ga}(s,\alpha) \}\]
we then define $\xi(t_0 - \Delta t)$ as an endpoint of the interval of the decomposition of $[0,|\ga|]$ related to $\mathcal S_\ga$, containing $s_0-\Delta t \alpha$, with
\[v(\xi(t_0- \Delta t), t_0 - \Delta t) \leq I_\ga[v](s_0-\Delta \al_{s_0}, t_0 - \Delta t).\]
We  therefore get
\[
   \left | \frac{s_0-\xi(t_0-\Delta t)}{\Delta t} - \alpha_{s_0} \right |\leq \frac{\Delta x}{\Delta t}, \quad | \xi(t_0 - \Delta t) - s_0| \leq \be_0 \, \Delta t + \Delta x  \]
   and
   \begin{eqnarray*}
   v \circ \ga(s_0,t_0) &=& I_\ga[v](s_0-\Delta \al_{s_0}, t_0 - \Delta t) + L_\ga(s_0,\al_{s_0}) \\
     &\geq& v(\xi(t_0- \Delta t), t_0 - \Delta t) + L_\ga \left (s_0, \frac{s_0-\xi(t_0-\Delta t)}{\Delta t} \right ) - \ell_0 \, \frac{\Delta x}{\Delta t},
   \end{eqnarray*}
where $\ell_0$ has been defined in Notation \ref{notmod}. If condition \eqref{nuovo1} still holds with $(\ga(\xi(t_0 - \Delta t)),t_0 - \Delta t)$ in place of $(x_0,t_0)$, we can iterate the above backward procedure, until we  determine  $t^* \in \T \cap \{t < t_0\}$, with $\xi(t^*) \in \{0,1\}$ or $t^* =0$, such that \eqref{nuovo1}  fails  with $(\ga(\xi(t^*)),t^*)$ in place of $(x_0,t_0)$.  We summarize in the next proposition the outputs of the above construction.

\begin{Proposition} Given $(x_0,t_0) \in \G^T$ with $x_0 \in \VV$, $t_0 >0$ such that \eqref{nuovo1} holds true for a suitable arc $\ga$ with $\ga^{-1}(x_0) = s_0$,  there exists  $t^* \in \mathcal T$, $t^* < t_0$ and  \[t \mapsto  (\xi(t), \alpha_{\xi(t)})\quad\hbox{from} \; \T \cap [t^*,t_0] \;\; \hbox {to} \;\; \SC_{\ga} \times \R, \]
  with $\xi(t_0)= s_0$,     satisfying
\begin{eqnarray}
 \left | \frac{\xi(t)-\xi(t-\Delta t)}{\Delta t} - \alpha_{\xi(t)} \right |&\leq &\frac{\Delta x}{\Delta t} \label{diffi-1}\\
 |\xi(t)- \xi(t- \Delta t)| &\leq&\be_0 \, \Delta t + \Delta x \label{diffi-11}
\end{eqnarray}
for $t \in [t^* + \Delta t,t_0] \cap \T_\Delta$, and
\begin{eqnarray}
  && \label{diffi1}  \hskip 1 cm   v \circ \ga(s_0,t_0) \geq v \circ \ga (\xi(t^*),t^*) \\  &+& \Delta t \, \sum_{i=1}^{  \lceil t_0 - t^*/ \Delta t\rceil } \left [  L_\ga \left ( \xi(t^* + i \, \Delta t), \frac { \xi(t^* + i \, \Delta t) - \xi(t^* + (i-1) \, \Delta t)} {\Delta t} \right ) - \ell_0 \, \frac {\Delta x} {\Delta t} \right ] \nonumber \\
   &\geq&   v \circ \ga (\xi(t^*),t^*) + \Delta t \, \sum_{i=1}^{  \lceil t_0 - t^*/ \Delta t\rceil } \left [  L_\ga \left ( \xi(t^* + i \, \Delta t), \frac { \xi(t^* + i \, \Delta t) - \xi(t^* + (i-1) \, \Delta t)} {\Delta t} \right ) \right ]  \nonumber \\ &-& T \, \frac {\Delta x} {\Delta t} .  \nonumber
\end{eqnarray}
\end{Proposition}
\smallskip
We call the $\xi(t)$, appearing in the previous statement, an {\it optimal trajectory}  for $ v(x_0,t_0)$.  We further  denote by $\widetilde \xi:[t^*,t_0] \to [0,|\ga|]$  the curve   obtained via linear interpolation of
  \begin{equation}\label{sanctuary}
  (t^* + i \Delta t,\xi(t^*+ i \Delta t) )    \qquad i=0, \cdots, \left\lceil{\frac{t_0-t^*}{\Delta t}}\right \rceil.
  \end{equation}

\smallskip
 The announced relevant estimate is provided in the next result.

\begin{Proposition}\label{coroptimum}
 We have
 \begin{eqnarray*}
  \int_{t^*}^{t_0} L_\ga(\wtd \xi,\dot{\wtd \xi}) \, dt &\leq & v\circ \ga(s_0,t_0) -  v \circ \ga(\xi(t^*),t^*) \\
    &+&  T \, \ell_0(\be_0 \, \Delta t + \Delta x + \Delta x/\Delta t),
 \end{eqnarray*}
 see Notation \ref{notmod} for the definition of $\ell_0$.
\end{Proposition}
\begin{proof} Taking into account   the very definition of $\widetilde \xi$, we have
\begin{equation}\label{diffi-0}
 \dot{\widetilde \xi}(t) = \frac{\xi(t^* + i\Delta t) -\xi(t^* +(i-1) \Delta t)}{\Delta t} \quad\hbox{for $t \in (t^* +(i-1) \Delta t, t^* + i \Delta t)$}
\end{equation}
which implies by \eqref{diffi-1}

\begin{eqnarray}
 \int_{t^*}^{t_0} L_\ga(\widetilde \xi, \dot{\widetilde\xi}) \, dt &=& \sum_{i=1}^{(t_0-t^*)/ \Delta t} \int_{ t^* + (i-1) \Delta t}^{ t^*+ i\Delta t} L_\ga(\widetilde \xi ,\dot{\widetilde\xi}) \, dt  \label{diffi4}\\
 &\leq& \Delta t \, \sum_{i=1}^{(t_0-t^*)/ \Delta t} L_\ga(\widetilde \xi (t^* + i \Delta t) ,\dot{\widetilde\xi}(t)) \, dt   \nonumber \\
 &+& T \, \ell_0 \, (\be_0 \, \Delta t + \Delta x) \nonumber
\end{eqnarray}

 By combining  \eqref{diffi1} and \eqref{diffi4}, and taking  into account \eqref{diffi-0}, we get the claimed inequality.
\end{proof}

\medskip

\subsection{Supersolution property}
The aim of this section is to prove the

\begin{Theorem}\label{teofreddo} The  function $u$ satisfies  the item {\bf (i)} in Definition \eqref{defsol}, with $H_\ga$ in place of $\wtd H_\ga$, at any vertex $x$.
\end{Theorem}

In the proof we essentially  exploit the definition of flux limiter through the following  result:

\begin{Lemma}\label{precrux} Let $\ga$ be an arc and $x$ a vertex incident on it. If $\xi:[a,b] \to [0,|\ga|]$ is a curve   with
$\xi(a)= \xi(b)=\ga^{-1}(x)$, then
\[\int_a^b L_\ga(\xi,\dot\xi) \, dt \geq c_x \, (b-a).\]
\end{Lemma}
\begin{proof} In \cite{PozzaSiconolfi}, see  Proposition 4.2., the statement is proved  for a Lagrangian satisfying the properties of $\wtd L_\ga$. For the proof in our case, it is the enough to recall that $L_\ga \geq \wtd L_\ga$, see {\bf (P3)}.
\end{proof}

\smallskip

We will also use the following:

\begin{Lemma}\label{fiori} Let $u^*$ be a Lipschitz continuous subsolution to (HJmod), then
\[u^* \circ \ga(s_1,t_1) - u^* \circ \ga (s_2,t_2) \leq \int_{t_1}^{t_2} L_\ga(\eta,\dot\eta) \, dt\]
for any $\ga$, $(s_i,t_i)$, $i=1,2$, with $t_1 <t_2$, any curve $\eta:[t_1,t_2] \to [0,|\ga|]$ joining $s_1$ to $s_2$.
\end{Lemma}
\begin{proof} The Lipschitz continuous function $u^* \circ \ga$ is subsolution to
\[u_t +  H_\ga(s,u')=0 \qquad\hbox{in $ (0,|\ga|) \times (0,T)$,}\]
with $H_\ga$ convex in the momentum variable.
 The statement is then an elementary fact in viscosity solutions theory. \end{proof}

We recall that the infinitesimal pairs $\Delta_m=(\Delta_m t, \Delta_m x)$ \ satisfy \eqref{opiccolo}.

\medskip

\begin{Lemma}\label{freddo} Let $x$ be a vertex, assume that
\begin{equation}\label{freddo1}
 \frac d{dt} \psi (t_0) < c_x
\end{equation}
for some $C^1$ subtangent $\psi$ to $u(x,\cdot)$ at $t_0 \in (0,T]$, then we find $ \de \in (0,t_0)$ and  sequences $t_m \in \T_m$  with $t_m \to t_0$ such that
\begin{itemize}
  \item[{\bf (i)}] $ v_m(x,t_m) -  v_m(x,t) \leq a \, (t_m-t)$ for any $ t \in \mathcal T_m \cap [t_0- \de, t_0)$, $m$  large enough,  and a suitable constant $a < c_x$;
  \item[{\bf (ii)}] there is  an arc $\ga\in \EE^+_x$ such  that
\begin{equation}\label{freddo2}
 v_m(x,t_m)= S_{\Delta,\gamma}[ v_m\circ \ga](\ga^{-1}(x),t_m - \Delta_m t)
\end{equation}
up to  subsequences.
\end{itemize}
\end{Lemma}
\begin{proof}
 By \eqref{freddo1}  there exist $\de \in (0,t_0)$
  and $a < c_x$ with
\[\psi(t_2) - \psi(t_1) \leq a \, (t_2-t_1) \qquad\hbox{for any $t_1$,   $t_2$ in $[t_0 - \delta, t_0+\delta] \cap [0,T]$.}\]
 We can  assume, without loosing generality, that $t_0$ is the unique minimizer of
 \[ v(x,\cdot) - \psi \qquad\hbox{in  $U := [t_0 - \delta, t_0 + \delta] \cap [0,T]$.}\]
Therefore any sequence $t_m$  of minimizers of $  v_m (x,\cdot) - \psi$ in $U \cap \mathcal T_m$ converges to $t_0$.
If $t  \in [t_0 - \de,t_0) \cap \mathcal T_m$ then
\[ v_m(x,t) - \psi(t) \geq   v_m(x,t_m) - \psi(t_m)\]
and consequently
\begin{equation}\label{diffi01}
   v_m(x,t_m) -  v_m (x,t) \leq \psi(t_m)- \psi(t) \leq a \, (t_m-t),
\end{equation}
for $m$ suitably large, which shows item {\bf (i)}.
Taking into account \eqref{diffi01},
 the  very definition of discrete problems, and the fact that the arcs in $\EE^+_x$ are finitely many, we derive that there exists  an arc $\ga \in \EE^+_x$  satisfying \eqref{freddo2} up to a subsequence.
\end{proof}

\smallskip
We can assume, without loosing generality, that the constant $\de$ appearing in the statement of Lemma \ref{freddo} satisfies
\begin{equation}\label{freddy0}
\de < \frac {|\ga|}{2 \be_0} \wedge \frac{|\ga|}2,
\end{equation}
To ease notations, we further assume  that $\ga^{-1}(x)=0$ and that the property {\bf (ii)} of the previous result holds for any $m$ and not just for a subsequence.
\smallskip

\begin{Lemma}\label{freddy} Let $\de$,  $\ga$ be the constant appearing in the previous lemma satisfying \eqref{freddy0}, and the arc for which \eqref{freddo2} holds true, respectively. Given $m \in \N$, let $[t^*_m, t_m]$ be the maximal interval of definition of an optimal trajectory $\xi_m$ for $ v_m \circ \ga(0,t_m)$, then $t^*_m \leq t_m -\de$ for $m$ suitably large.
\end{Lemma}
\begin{proof}  We first point out that optimal trajectories $\xi_m$ for $ v \circ \ga(0,t_m)$ do  exist because of condition \eqref{freddo2}. We fix $m$ and assume by contradiction that $t^*_m > t_m-\de >0$, by exploiting the fact that $\xi_m(t_m)=0$, \eqref{diffi-11} and \eqref{freddy0}, we get
\begin{eqnarray*}
  |\xi_m(t^*_m)| &\leq& \be_0 \, (t_m -t^*_m) + \frac{t_m -t^*_m}{\Delta_m t} \, \Delta_m x \leq \left (\be_0  + \frac{\Delta_m x}{\Delta_m t} \right ) \, \de \\
   &<& \left (\frac 12 + \frac 12 \right ) \, |\ga|.
\end{eqnarray*}
Therefore $\xi_m(t^*_m) \neq |\ga|$, taking into account that $\xi_m(t^*_m) \in \{0,|\ga|\}$ or $t_m^*=0$, we deduce that $\xi_m(t^*_m)=0$. We denote by $\wtd \xi_m$ the trajectory supported in $[0,|\ga|]$ obtained by linear interpolation from $\xi_m$, see \eqref{sanctuary}.  We derive from Proposition  \ref{coroptimum} and the condition {\bf (i)} of Lemma \ref{freddo}
\[ \int_{t^*_m}^{t_m} L_\ga(\wtd \xi_m, \dot{\wtd \xi_m}) \, dt
<  a \, (t_m - t^*_m) + T \, \ell_0(\be_0 \, \Delta_m t + \Delta_m x + \Delta_m x/\Delta_m t)\]
which implies
\[ \int_{t^*_m}^{t_m} L_\ga(\wtd \xi_m, \dot{\wtd \xi_m}) \, dt < c_x \, (t_m-t_m^*),\]
provided that $m$ is large enough. This is in contradiction with  Lemma \ref{precrux}, and proves that $t^*_m \leq t_m -\de$ for $m$ large enough, as was asserted.
\end{proof}

\begin{proof}[{\bf Proof of Theorem \ref{teofreddo}}] We are in the setup described by the statement of Lemma ref{freddo}. We keep the notation $\xi_m$, $\wtd \xi_m$ for the sequences of curves introduced in the previous proof. To repeat: $\xi_m$ is an optimal trajectory for  $ v_m \circ \ga(0,t_m)$, $\wtd \xi_m$ is the curve obtained from $\xi_m$ by linear interpolation. Given $\eps >0$, we focus on the $m$'s with $t_m < t_0+\eps$, $t_m < t_0+\de/2$, where $\de$ is the constant appearing in Lemma \ref{freddo}.

Thanks to Lemma \ref{freddy}  the $\wtd \xi_m$'s are  defined in the interval $[t_0 - \de/2,t_m]$, we extend them, without changing notation,  to $[t_0-\de/2,t_0+ \eps]$ by setting
\[\wtd \xi_m (t)= \wtd \xi_m(t_m)=0 \qquad\hbox{for every $m$, $t \in (t_m,t_0 +\eps]$.}\]
The curves $\wtd \xi_m$ so extended are  equibounded and equiLipschitz continuous by \eqref{diffi-1}, \eqref{diffi-0}, and consequently uniformly convergent in $[t_0 - \de/2,t_0+ \eps]$, up to a subsequence, to a limit curve   $\wtd \xi$. By the lower semicontinuity of the action functional with respect to the uniform convergence,  see for instance \cite{Buttazzo}, we get
\begin{equation}\label{teofreddo1}
\liminf_m \int_{t_0-\de/2}^{t_0+\eps} L_\ga(\wtd \xi_m, \dot{\wtd \xi_m}) \, dt \geq
\int_{t_0 - \de/2}^{t_0+\eps} L_\ga(\wtd \xi, \dot{\wtd \xi}) \, dt.
\end{equation}
Let $\de_m$ be a sequence with $\de_m <  \frac \de 2$, $t_0- \de_m \in \mathcal T_m$, $\de_m \to \frac \de 2$, we have by Proposition \ref{coroptimum}
\begin{eqnarray*}
 && \int_{t_0-\de/2}^{t_0+\eps} L_\ga(\wtd \xi_m,\dot{\wtd \xi_m}) \, dt =
\int_{t_0-\de/2}^{t_0-\de_m} L_\ga(\wtd \xi_m,\dot{\wtd \xi_m}) \, dt + \int_{t_0-\de_m}^{t_m} L_\ga(\wtd \xi_m,\dot{\wtd \xi_m}) \, dt\\ &+&  \int_{t_m}^{t_0+\eps} L_\ga(\wtd \xi_m,\dot{\wtd \xi_m}) \, dt
   \leq  v_m\circ \ga(0,t_m) - v_m \circ \ga(\wtd\xi_m(t_0-\de_m),t_0-\de_m)\\ &+& (t_m-t_0+\de_m) \, \ell_0(\be_0 \, \Delta_m t + \Delta_m x + \Delta_m x/\Delta_m t) + M \, (\de/2- \de_m +t_0+\eps -t_m ),
\end{eqnarray*}
where $\ell_0$ has been defined in Notation \ref{notmod} and  $M$ is an upper bound of $L_\ga$ in $[0,|\ga|] \times [-\be_0, \be_0]$. Passing at the limit as $m \to + \infty$ and taking into account \eqref{teofreddo1}, we get
\[\int_{t_0 - \de/2}^{t_0+\eps} L_\ga(\wtd \xi, \dot{\wtd \xi}) \, dt \leq  u\circ \ga(0,t_0) - u \circ \ga(\wtd\xi(t_0-\de),t_0-\de) + M \, \eps, \]
 sending $\eps$ to $0$, we further get
 \[\int_{t_0 - \de/2}^{t_0} L_\ga(\wtd \xi, \dot{\wtd \xi}) \, dt \leq  u\circ \ga(0,t_0) - u \circ \ga(\wtd\xi(t_0-\de),t_0-\de). \]
 Since $u\circ \ga $ is a subsolution to \eqref{HJgm} by Proposition \ref{main}, \ref{bimain}, equality must prevail in the above formula by Lemma \ref{fiori}. Now, assume for purposes of contradiction that there is a $C^1$ subtangent $\phi$, constrained  to $[0,|\ga|] \times [0,T]$, to $ u \circ \ga$ at $(0,t_0)$ with
  \[ \phi_t(0,t_0) + H_\ga(0,\phi'(0,t_0))
  < 0,\]
 then we can construct via Perron--Ishii method a new subsolution of \eqref{HJg}, say $w: [0,|\ga|] \times [0,T] \to \R$,   with
$w$ strictly greater than $u \circ \ga$ locally at $(0,t_0)$ and agreeing with $u \circ \ga$ at $\wtd \xi(t_0-\de,t_0-\de)$,  so that
\begin{eqnarray*}
  w(0,t_0) - w(\wtd\xi(t_0-\de),t_0-\de) & >& u\circ \ga(0,t_0) - u \circ \ga(\wtd\xi(t_0-\de),t_0-\de) \\
   &=& \int_{t_0 - \de}^{t_0} L_\ga(\wtd \xi, \dot{\wtd \xi}) \, dt,
\end{eqnarray*}
which is impossible by Lemma \ref{fiori}.
\end{proof}

\bigskip

\section{main convergence theorem}

\begin{Theorem}\label{finale}  Let $\Delta_m$ a sequence of pairs satisfying \eqref{opiccolo}. The whole sequence $v_m$ uniformly converges to the unique solution of the original problem (HJ).
\end{Theorem}
\begin{proof} Let $u$ be a limit function of $v_m$. By Propositions \ref{main},  \ref{bimain} and Theorem \ref{teofreddo}, $u$ is solution to (HJmod) and, since such a solution is unique by Theorem \ref{exunmod}, we derive that the whole sequence $v_m$ uniformly converges to $u$. Finally, by Theorem \ref{modteo}, $u$ is solution to (HJ).

\end{proof}

\bigskip

\section{Numerical simulations}\label{simu}
In this section, we describe some numerical simulations in the case of networks in $\R^2$, to show the performance and the convergence properties of the  proposed scheme \eqref{eq:scheme}.
We measure the accuracy of the scheme by computing the following  uniform and 1-norm of the errors:
\begin{equation}\label{err}
E^\infty := \max_{ (x,T) \in \Gamma^T_{\Delta}} |v(x,T) -  v_{\Delta}(x,T)|, \quad E^1 := \sum_{( x,T) \in \Gamma^T_{\Delta}} |v(x,T) -  v_{\Delta}(x,T)| \Delta x.\end{equation}

We perform two numerical tests. In the first example we consider a simple network, and two different  Hamiltonians. For the first one, continuous on the whole network and not depending on $s$,  the analytical solution $v$ of the time--dependent problem is known, so that the errors   between the approximated solution $v_{\Delta}$  and $v$, can be exactly computed.  For the second   Hamiltonian, discontinuous and  depending  on $s$, for which the analytical solution   is not available, we compare $v_{\Delta}$ with a reference solution computed on a much finer grid, which we regard as the exact solution.

In the second example, we consider a more complex network  and  two Hamiltonians: one continuous and the other discontinuous.
The numerical simulations have been carried out in MatLab R2020 on a machine with Processor $2,5$ GHz Intel Core $i7$.

As discussed in \cite{FF15}, the truncation error for first order semi--Lagrangian scheme applied to Hamilton--Jacobi equation in $\R$   with continuous Hamiltonian is  $\mathrm O(\frac{\Delta x^2}{\Delta t}+\Delta t)$. In our case, we have an additional error due to the fact that   the local equations is not posed in the whole of $\R$ but in a bounded interval, so that the controls must be bounded as well. However the error so introduced  is of order $O(\Delta t)$, so that, in the end, the above estimate $\mathrm O(\frac{\Delta x^2}{\Delta t}+\Delta t)$ is still valid in our context.
It is worth pointing out that the additional error does not appear in \cite{carlini20}. This is, in some sense, the price we are paying for having reduced the computation, as much as possible, to the local equations defined on any single arc of the network. On the other hand the reduction allows  a large gain in computational time (see Table \ref{tab:test1a}-- \ref{tab:test2a}).

The above estimate of the truncation error shows that  the choice $\Delta t=O(\Delta x)$  is optimal, and heuristically indicates that   convergence rate is 1.
However, our theoretical convergence analysis requires $\frac{\Delta x}{\Delta t}\to 0$, we have consequently  chosen a balance between the two parameters compatible with  the theoretical results
and at the same time not too far from the optimal
case,  we have therefore set  $\Delta t=C \Delta x^{4/5}$, with $C$ positive constant.

We finally notice that for simplicity in the implementation, we have used a uniform space grid, which is slightly different  from what done in the  theoretical part.

\subsection{Test on a simple network}\label{sec:test1}\
We consider as $\Gamma \subset \R^2$  the triangle  with vertices
\[v_1 =(0,0), \qquad v_2=(1,0),  \qquad  v_3=(1/2,1/2),\]
with arcs
\[\ga_1(s)=  s\, v_2,  \quad \ga_2(s)=s \, v_3, \quad  \ga_3(s)= (1-s) \, v_2+ s\, v_3,\]
plus the  reversed arcs. Clearly $\{\ga_1, \ga_2, \ga_3\}$ make up an orientation of $\Gamma$.\\
{\em{Hamiltonian  independent of $s.$}} We choose as cost functions   $L_{\gamma_i}(s,\la)=\frac{\la^2}{2}$,  for all $s\in [0,|\ga_i|]$,
admissible flux limiters $c_i=-5$
for $i=1,2,3$,
 and as initial condition $g=0$. The exact solution, for $t\geq \frac{1}{\sqrt{10}}$,
 has the following form
$$
v(x,t)=\begin{cases}(0.5-\left |x_2-0.5\right |)\sqrt{10}-5t   &{\mbox {if}\;}  x\in \gamma_1,\vspace{0.2cm} \\
\left(\frac{1}{2\sqrt{2}}-\left||x|-\frac{1}{2\sqrt{2}}\right|\right)\sqrt{10}-5t       &{\mbox {if}\;}  x\in \gamma_2, \vspace{0.2cm} \\
\left(\left||x-(\frac{1}{2},\frac{1}{2})|-\frac{1}{2\sqrt{2}}\right|+\frac{1}{2\sqrt{2}}\right)\sqrt{10}-5t       &{\mbox {if}\;}  x\in \gamma_3,
\end{cases}
$$
where $x=(x_1,x_2)$ is the variable in the physical space $\R^2$. We compute errors in the uniform and 1-norm, as defined in  \eqref{err}, at $T=1$, and computational times,  for several refinements of the grids  with $\Delta t=\Delta x^{4/5}/2$.
The  hyperbolic CFL condition $\max_{\gamma,s}|u'(\gamma(s))|\Delta t \leq \Delta x$ (see (1.9) in \cite {costeseque2015convergent}) is not satisfied, and  large Courant number $\nu=\underset{\gamma,s}\max|u'(\gamma(s))|\frac{\Delta t}{\Delta x}=\frac{\sqrt{10}}{2\Delta x^{1/5}}>1$ are allowed.
We show the results in Table \ref{tab:test1a},  we further
show  there errors and computational  times for  the scheme proposed  in \cite{carlini20},  denoted by SL. We observe that the new scheme has rate of convergence close to 0.5  in the uniform norm and close to 1 in the 1 norm  (columns 2 and 3). Compared to the SL scheme (columns 5 and  6)  the accuracy is lower, but the computational time is smaller (column 4 vs column 7).

We remark that  in this case the  SL scheme computes the exact solution, since the characteristics are affine,  the truncation error is only due to minimization.

We also consider  the balance $\Delta t=\Delta x/2$, which is out of our theoretical convergence hypotheses, and we find that also in this case  the scheme has a convergent behavior.
In  Table \ref{tab:test1b}, we observe order one of convergence in both norms (columns two and three). Convergence rate improves at the price of higher computational cost, with respect the case  $\Delta t=\Delta x^{4/5}/2$. Compared to the SL scheme (columns 5 and 6) we observe bigger errors and still smaller computational time (columns 4 and 7).

\begin{table}[h]
\begin{small}
\centering
\begin{tabular}{lllllllllll}
\hline
 ${\Delta x}$ &$E^{\infty}$ &$E^{1}$     &time&$E^{\infty}$   &$E^{1}$     &time\\
 \hline
 $1.00\cdot 10^{-1}$   &$1.62\cdot 10^{-1}$    &$3.99 \cdot 10^{-2}$  & $0.03s$   &$1.19\cdot 10^{-5}$      &$1.11\cdot 10^{-6}$ &$0.90s$ \\
$5.00\cdot 10^{-2}$    &$1.18\cdot 10^{-1}$    &$2.14 \cdot 10^{-2}$ &$0.14s$  &$9.03\cdot 10^{-7}$          &$2.49 \cdot 10^{-7}$  &$5.07s$\\
$2.50\cdot 10^{-2}$    &$6.52\cdot 10^{-2} $   &$7.86 \cdot 10^{-3}$       &$0.62s$   &$4.13\cdot 10^{-7} $      &$1.14\cdot 10^{-7} $  &$35.8s$  \\
$1.25\cdot 10^{-2}$     &$4.16\cdot 10^{-2} $   &$3.65 \cdot 10^{-3}$    &$3.63s$  &$7.62\cdot 10^{-8} $   &$1.46\cdot 10^{-8} $     &$\,358s$\\
$6.25\cdot 10^{-3}$     &$2.67 \cdot 10^{-2}$  &$1.58 \cdot 10^{-3}$   & $25.6s$ &$3.29\cdot 10^{-8} $ &$8.09\cdot 10^{-9} $         &$2895s$\\
\hline
\end{tabular}
\end{small}
\vspace{0.2cm}
\caption{
Errors and convergence rate  for example in Section \ref{sec:test1} with quadratic cost independent on $x$, computed at  time $T=1$ with $\Delta t=\Delta x^{4/5}/2$. Columns 2-4 refer to the  the new scheme. Columns 5-7 refer to  the SL scheme  \cite{carlini20} \label{tab:test1a}.}
\end{table}

\begin{table}[h]
\begin{small}
\centering
\begin{tabular}{lllllllllll}
\hline
 ${\Delta x}$ &$E^{\infty}$ &$E^{1}$     &time&$E^{\infty}$   &$E^{1}$     &time\\
 \hline
 $1.00\cdot 10^{-1}$   &$3.37\cdot 10^{-2}$              &$1.38 \cdot 10^{-2}$  & $0.05s$ &$1.19\cdot 10^{-5}$          &$1.01\cdot 10^{-6}$  &$0.40s$\\
$5.00\cdot 10^{-2}$    &$1.68\cdot 10^{-2}$              &$6.58 \cdot 10^{-3}$ &$0.24s$  &$9.03\cdot 10^{-6} $      &$2.49\cdot 10^{-7} $  &$3.57s$  \\$2.50\cdot 10^{-2}$    &$8.44\cdot 10^{-3} $      &$3.25 \cdot 10^{-3}$       &$1.21s$   &$4.13\cdot 10^{-7} $   &$1.14\cdot 10^{-7} $     &$25.2s$\\$1.25\cdot 10^{-2}$     &$4.22\cdot 10^{-3} $   &$1.61 \cdot 10^{-3}$    &$8.92s$  &$7.62\cdot 10^{-8} $ &$1.46\cdot 10^{-8} $ &$\,174s$\\
$6.25\cdot 10^{-3}$     &$2.11 \cdot 10^{-3}$  &$8.15 \cdot 10^{-4}$   & $71.1s$ &$3.29\cdot 10^{-8} $ &$8.09\cdot 10^{-9} $ &$1728s$\\

\hline
\end{tabular}
\end{small}
\vspace{0.2cm}
\caption{
Errors  and convergence rate  for example in Section \ref{sec:test1} with quadratic cost independent on $x$, computed at  time $T=1$ with $\Delta t=\Delta x/2$. Columns 2-4 refer to the  the new scheme. Columns 5-7 refer to  the SL scheme  \cite{carlini20} \label{tab:test1b}}
\end{table}

\begin{figure}[htbp]
\centering
	\includegraphics[width=0.5\textwidth]{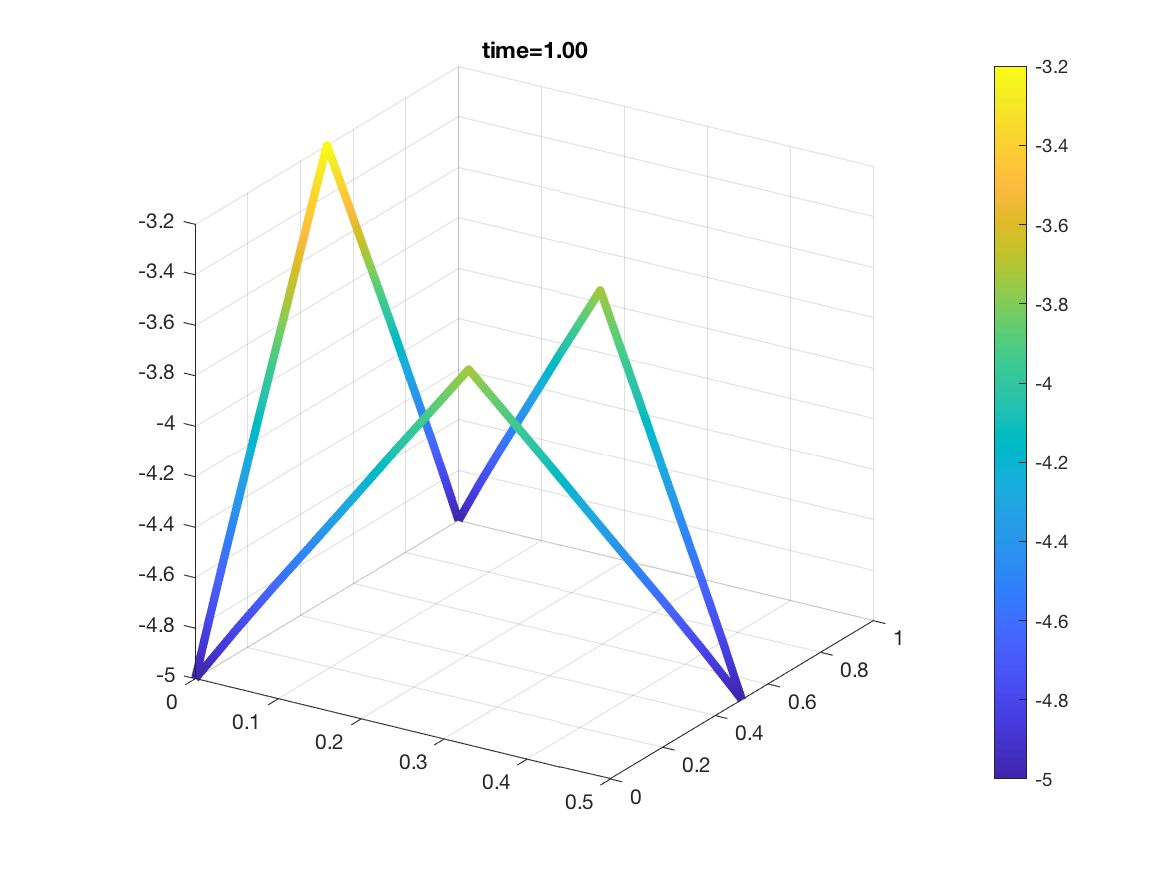}\includegraphics[width=0.5\textwidth]{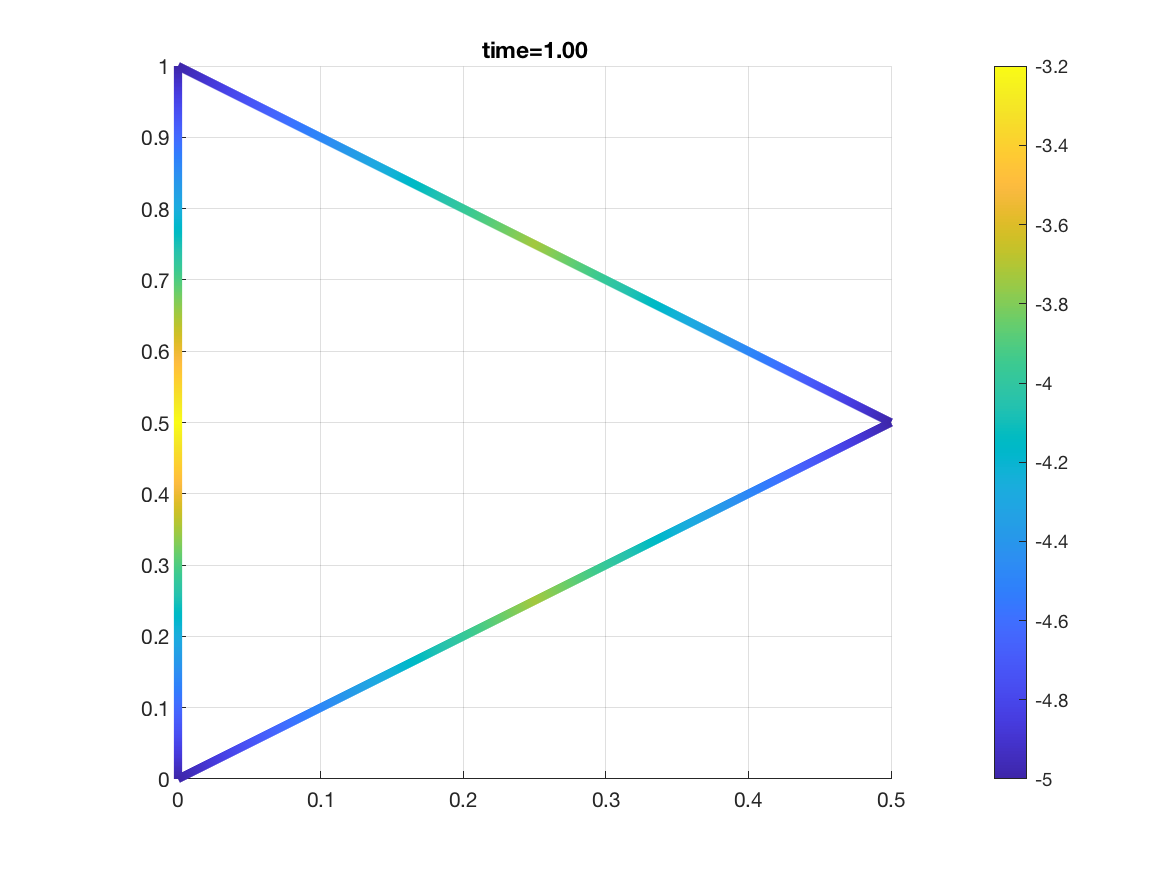}
\caption{Example in  Section~\ref{sec:test1} with  quadratic Hamiltonians and   $c_1=c_2=c_3=-5$. Solution $v_{\Delta}$ at  time $T=1$ computed with $\Delta x =0.05$, $\Delta t=\frac{\Delta x}{2}$
 \label{fig:test1}
  (Left plot: 3d view. Right plot: 2d view on $x_1x_2$ plane).  }
\end{figure}

\begin{figure}[htb!]
\centering
	\includegraphics[width=0.5\textwidth]{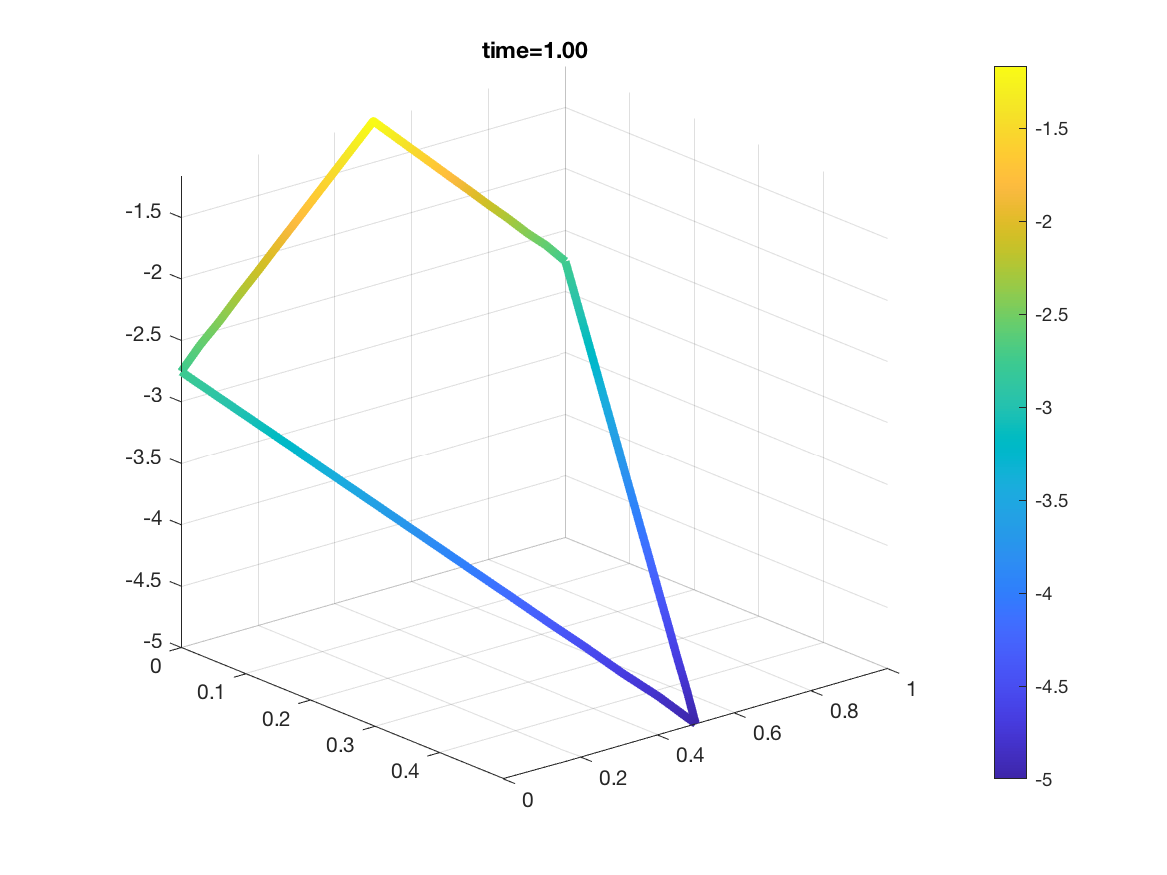}\includegraphics[width=0.5\textwidth]{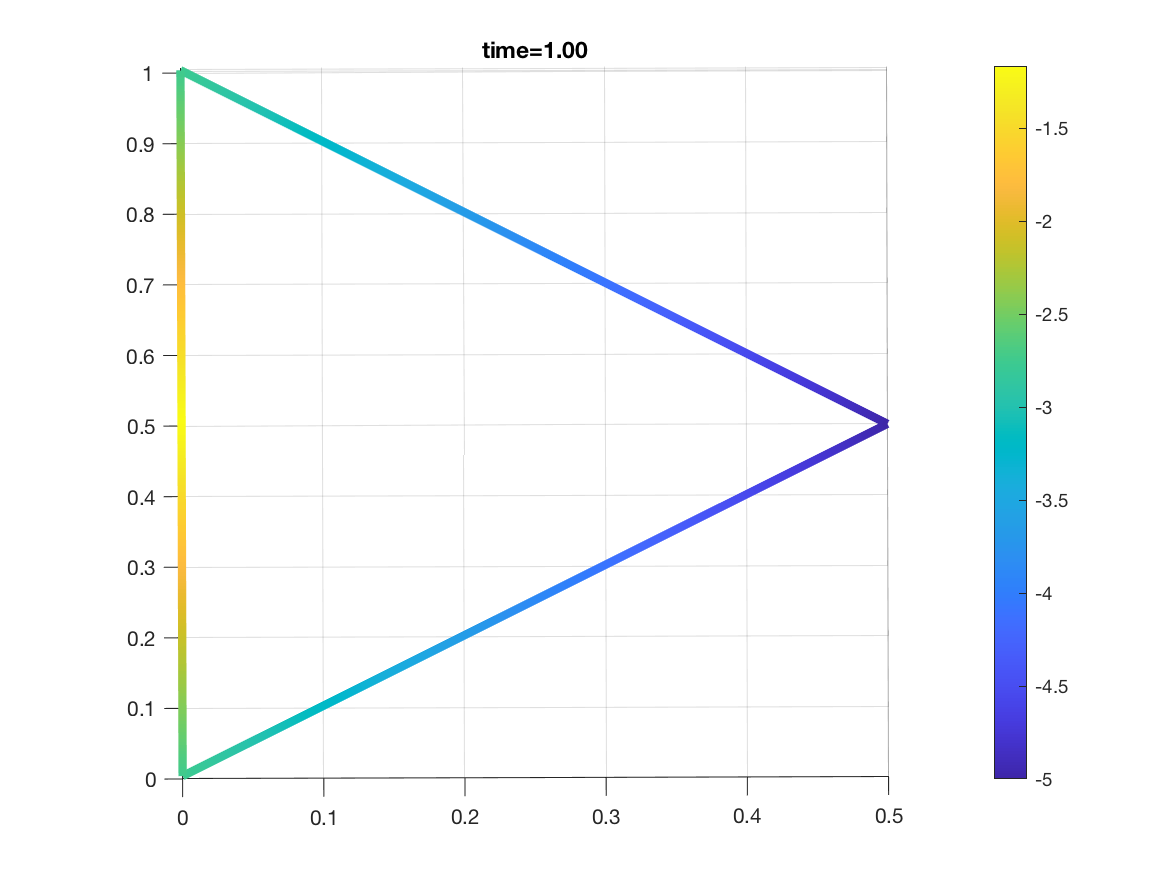}
\caption{Example in  Section~\ref{sec:test1} with  quadratic Hamiltonians and $c_1=c_2=-1,\;c_3=-5$. Solution $v_{\Delta}$ at  time $T=1$
computed with $\Delta x =0.05$,  $\Delta t=\frac{\Delta x}{2}$.
(Left plot: 3d view. Right plot: 2d view on $x_1x_2$ plane).  \label{fig:test1b}}
\end{figure}
In Figure \ref{fig:test1}, we show the approximated solution at time $T=1$ computed with $\Delta x=0.05$ and $\Delta t=0.025$.
In Figure \ref{fig:test1b}, we show the approximated solution at time $T=1$ obtained setting flux limiters   $c_1=c_2=-1$, and  $c_3$ still equal to  $-5$.  With this choice, it is more convenient to choose characteristics starting from  $v_3$, because of the small value of  the flux limiter at this vertex.\\

{\em{Hamiltonian  depending on $s.$}} Let us now choose cost functions depending on $s$, as
\begin{eqnarray*}
L_{\ga_1}(s,\la)  &=& \frac{|\la|^2}{2}+ 5|(\ga_1(s)-(0.5,0.5)|^2 \\
L_{\gamma_2}(s,\la) &=& \frac{|\la|^2}{2}+ 5|\ga_2(s) -(0.5,0.5)|^2+ 10 (\ga_1(s)_1)^2  \\
L_{\gamma_3}(s,\la) &=&\frac{|\la|^2}{2}+ 5|\ga_3(s)-(0.5,0.5)|^2+10 (\ga_3(s)_1)^2,
\end{eqnarray*}
where $\ga(s)_1$ indicates the first component of the point $\ga(s)$. The cost in  $\gamma_1$  is optimized at  $v_3$, whereas
the costs in  $\gamma_2$ and  $\gamma_3$  increase with the distance of the physical point from $v_3$ and  from the line $x_1=0$.
In Figure \ref{fig:test1c}, we show the initial condition (left) toghether with the approximated solution computed at $T=1$ (center and right).
We choose as admissible flux limiters $c_1=c_2=c_3=1$.
\begin{figure}[htbp]
\centering
	\includegraphics[width=0.33\textwidth]{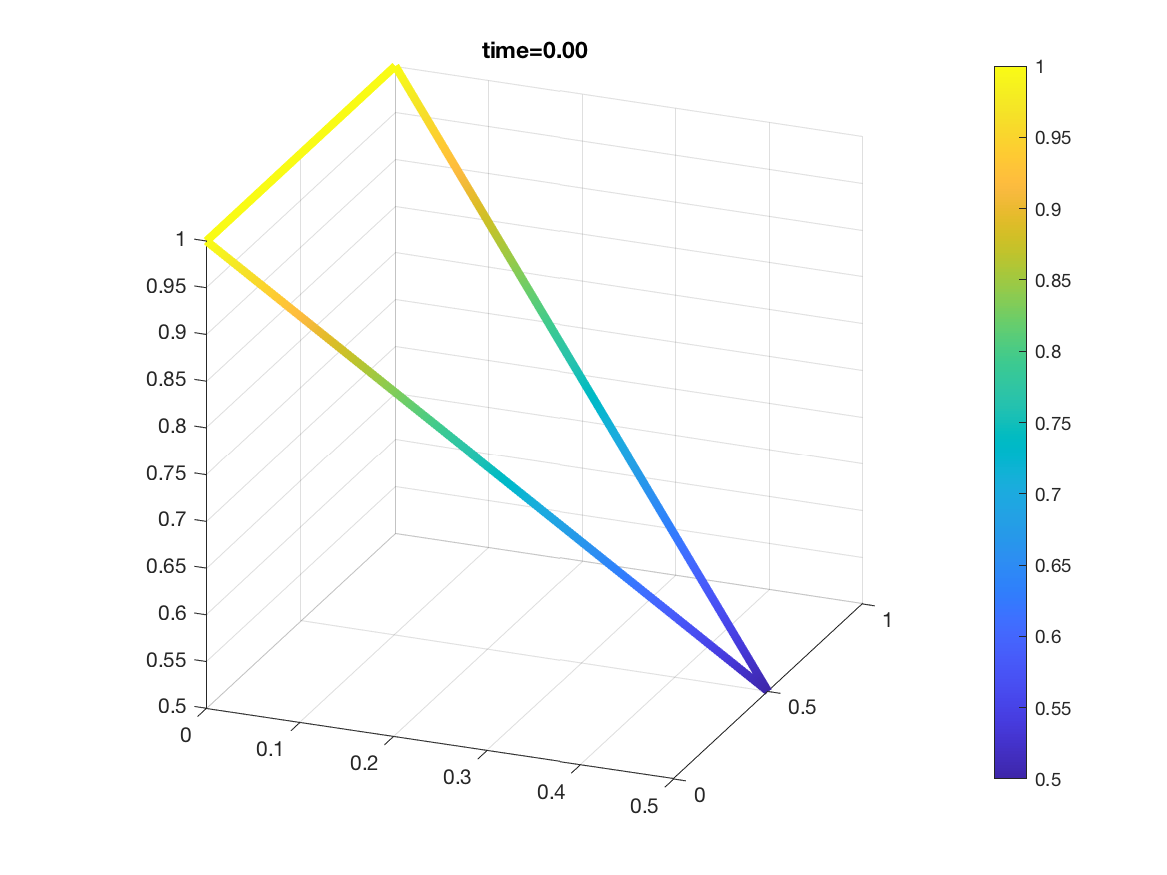}\includegraphics[width=0.33\textwidth]{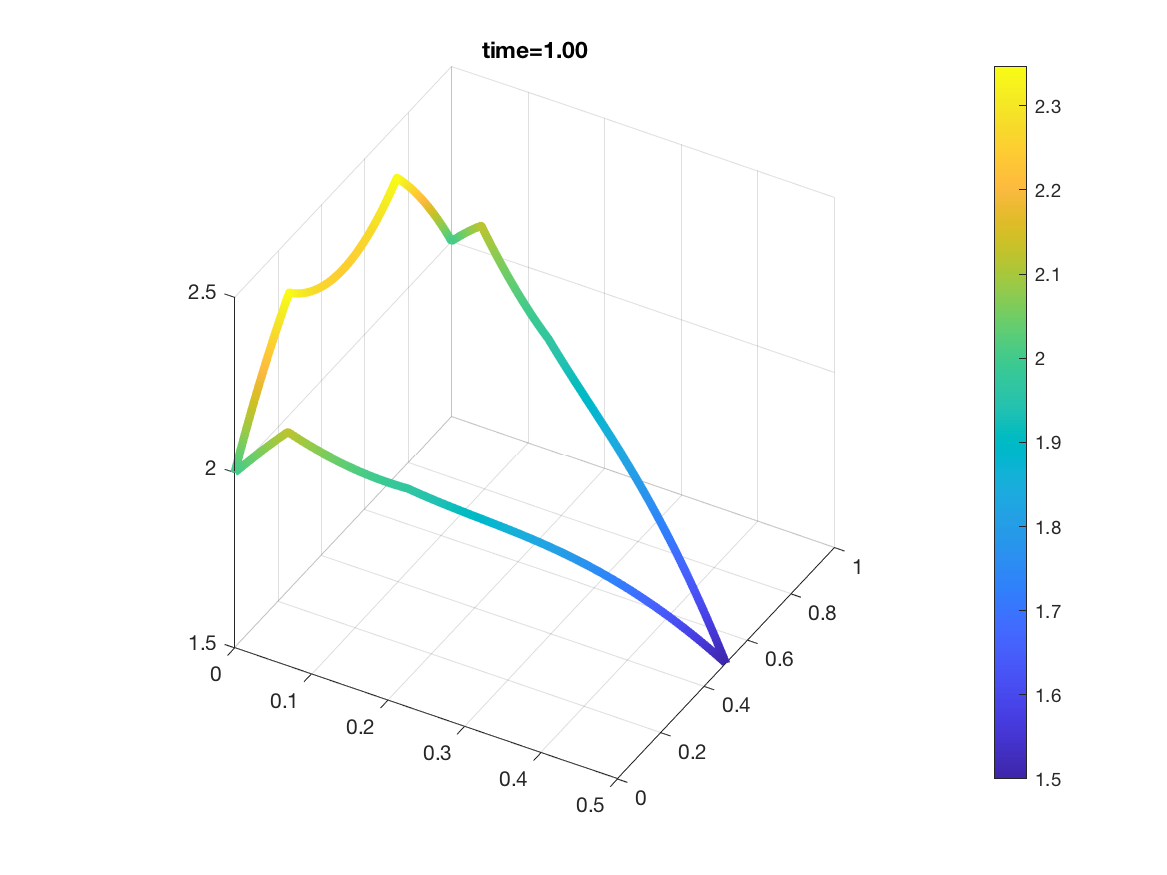}\includegraphics[width=0.33\textwidth]{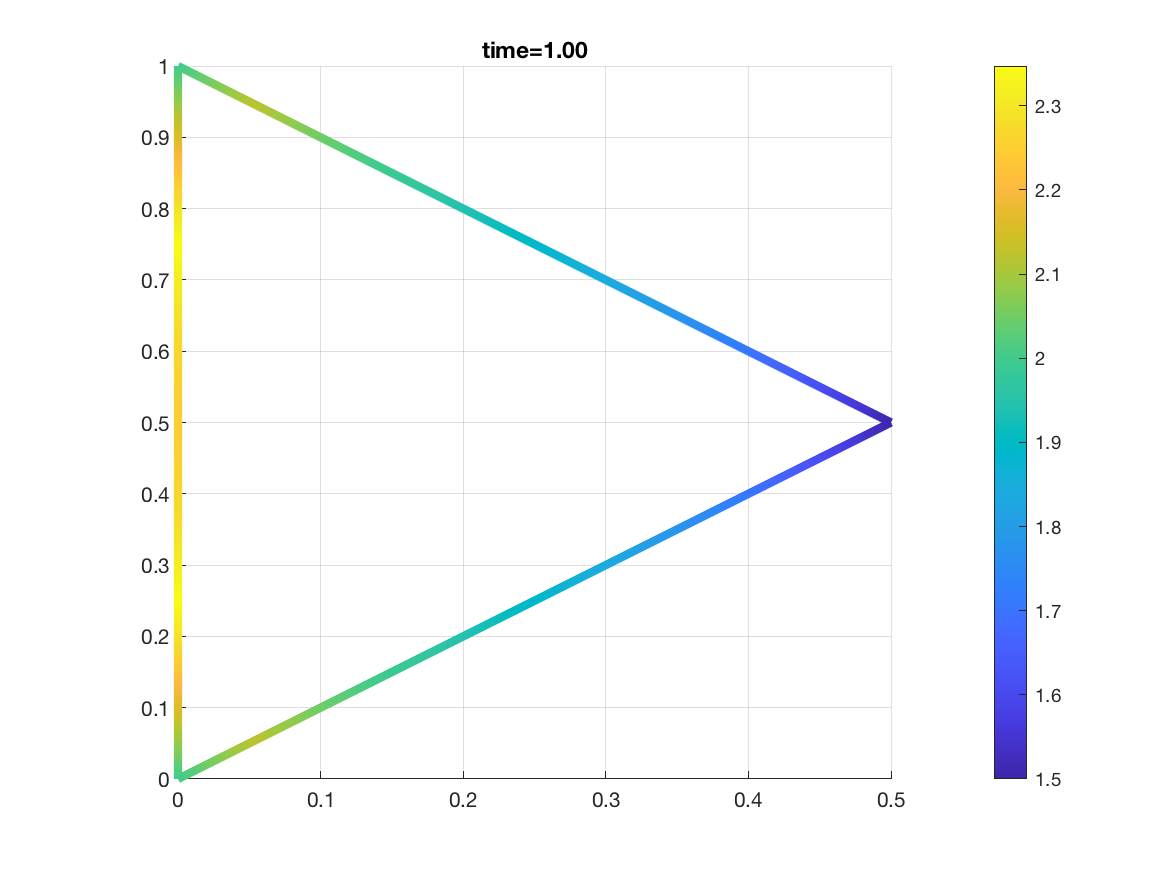}
\caption{Example in  Section~\ref{sec:test1} with Hamiltonian depending on $s$. Initial condition (left) and approximated solution (center, right) at final time $T=1$ with $c_1=c_2=c_3=1$, $\Delta x =6.25 \cdot 10^{-2}$ and $\Delta t=\Delta x^{4/5}/2$  \label{fig:test1c}.  }
\end{figure}
We compute a reference solution using \eqref{eq:scheme} with $\Delta x=10^{-3}$ and  $\Delta t=\Delta x^{4/5}/2$.
In Table \ref{tab:test2a}, the errors and computation times for scheme \eqref{eq:scheme} and the SL scheme are shown.
  We remark that when the characteristics are not affine, both schemes have order one of consistency, and so the new scheme is advantageous in terms of computation time, without any loss of accuracy.
Note, finally, that there is  an almost $99\%$ reduction in computation time for  comparable errors.

\begin{table}[h]
\begin{small}
\centering
\begin{tabular}{lllllllllll}
\hline
 ${\Delta x}$ &$E^{\infty}$ &$E^{1}$     &time&$E^{\infty}$   &$E^{1}$     &time\\
 \hline
 $1.00\cdot 10^{-1}$   &$4.58\cdot 10^{-2}$              &$5.69 \cdot 10^{-2}$  & $0.04s$   &$4.66\cdot 10^{-2}$      &$5.85\cdot 10^{-2}$ &$0.87s$ \\
$5.00\cdot 10^{-2}$    &$2.50\cdot 10^{-2}$              &$2.83 \cdot 10^{-2}$ &$0.18s$  &$2.81\cdot 10^{-2}$          &$3.35\cdot 10^{-2}$  &$5.22s$\\
$2.50\cdot 10^{-2}$    &$1.34\cdot 10^{-2} $      &$1.61\cdot 10^{-2}$       &$0.68s$   &$1.57\cdot 10^{-2} $      &$1.89\cdot 10^{-2} $  &$39.2s$  \\
$1.25\cdot 10^{-2}$     &$7.67\cdot 10^{-3} $   &$9.30 \cdot 10^{-3}$    &$3.94s$  &$8.61\cdot 10^{-3} $   &$1.03\cdot 10^{-2} $     &$306s$\\
$6.25\cdot 10^{-3}$     &$4.01\cdot 10^{-3}$  &$1.48 \cdot 10^{-3}$   & $27.5s$ &$4.44\cdot 10^{-3} $ &$5.24\cdot 10^{-3} $         &$2664s$\\
\hline
\end{tabular}
\end{small}
\vspace{0.2cm}
\caption{
Errors, computed at  time $T=1$ with $\Delta t=\Delta x^{4/5}/2$, and convergence rate  for example in Section \ref{sec:test1} with Hamiltonian  dependent on $x$. Columns 2-4 refer to the  the new scheme. Columns 5-7 refer to  the SL scheme \cite{carlini20}. \label{tab:test2a}}
\end{table}
\subsection{Test on a Traffic circle}\label{sec:Test2}
We consider as $\Gamma \subset \R^2$ a network formed by 8 vertices and 12 arcs. The vertices are defined as $v_1 =(-2,0), \; v_2=(-1,0),  \;  v_3=(0,2),\; v_4 =(0,1), \; v_5=(2,0),  \;  v_6=(1,0),\;v_7=(0,-2),  \;  v_8=(0,-1)$. The  arcs are defined, for $s\in [0,1]$, as
$$\begin{array}{crc}
 \ga_1(s)= (1-s)  v_1+ s v_2, &\ga_2(s)= (1-s)  v_1+ s  v_3,  &\ga_3(s)= (1-s) v_1+ s v_7,\\
 \ga_4(s)= (1-s)  v_2+ s v_4, &\ga_5(s)= (1-s) v_2+ s v_8,&\ga_6(s)= (1-s)  v_3+ s  v_4,\\
\ga_7(s)= (1-s)  v_3+ s v_5,&\ga_8(s)= (1-s)  v_4+ s v_6,&\ga_9(s)= (1-s)  v_5+ s v_6,\\
\ga_{10}(s)= (1-s) v_5+ s v_7,&\ga_{11}(s)= (1-s) v_6+ sv_8,&\ga_{12}(s)= (1-s) v_7+ s  v_8,
\end{array}$$
plus the reversed arcs. Let  $\{\gamma_1,\dots,\gamma_{12}\} \in {\bf{E^+}} \subset \Gamma$.
For any $\gamma \in \mathbf{E}^+$, consider the following Hamiltonians, $$H_{\ga}(s,\mu)=\frac{1}{2}|\mu|^2-|\gamma(s)-\bar x|^2,$$
where  $\bar x=(1,1)$ plays the role of a target, since the cost reaches its minimum on $\bar x$.
For $i\in \{1,\dots,8\}$, we call $c_{i}$ the flux limiter  on the  vertex $v_i$, which in order to be admissible must be such that  $c_1\leq 2$, $c_2\leq 1$, $c_3\leq 0$, $\leq c_4\leq 0.5$, $c_5\leq 0$, $c_6\leq 0.5$, $c_7\leq 2$, and $ c_8\leq 1$.
We take $g=0$ as the initial condition, set all flux limiters equal to their maximum admissible value, and calculate the solution of the scheme \eqref{eq:scheme}
with $\Delta t=(\Delta x)^{4/5} / 2$, $\Delta x=0.025$. 
  In Figure \eqref{fig:test2a}, we show the approximated solution at $T=2$ and $T=5$.  At $T=2$, we  observe  local minima reached  at all the vertices,  at the target and at $(0.5,0.5)$. Although one might expect a single global minimum to be reached at the target for large time,  we instead observe   three global minima at $T=5$,  at the target and  at the adjacent  two vertices.

  We proceed considering the case with two targets, the  Hamiltonian is now discontinuous. With respect to the previous example, we change  the definition of the Hamiltonian in the arcs of the  inner circle, and in the arcs connecting  inner  and outer circle, taking into account the new target $\tilde x =(-0.5,-0.5)$ and  setting
 $$H_{\ga}(s,\mu)=\frac{1}{2}|\mu|^2-|\gamma(s)-\tilde x|^2.$$
  The compatibility conditions for   flux limiters are
  $c_1\leq 0.5$, $c_2\leq 0$, $c_3\leq 0$, $c_4\leq 0.5$, $c_5\leq 0$, $c_6\leq 0.5$, $c_7\leq 0.5$, and $c_8\leq 0$.
The graphs, at the top of  Figure \ref{fig:test2b},   show the approximated value function, computed at $T=2$ with flux limiters  equal to the maximum admissible value.  Also in this case, global minima are reached at the target points and at the adjacent vertices. By choosing a lower value for the flux limiter relative to the vertices adjacent to the target points, we set e.g.  $c_2=c_3=c_4=c_8=-1$,
  the  approximated value function, shown  at the bottom of
  Figure \ref{fig:test2b}, reaches the global minima only at the vertices adjacent to the  targets.
  This  example  highlights the artificial effect introduced by flux limiters, especially the non--maximal ones, in the  control problem, where optimal trajectories are  expected to start from target points.
 \begin{figure}[htbp]
\centering
	\includegraphics[width=0.5\textwidth]{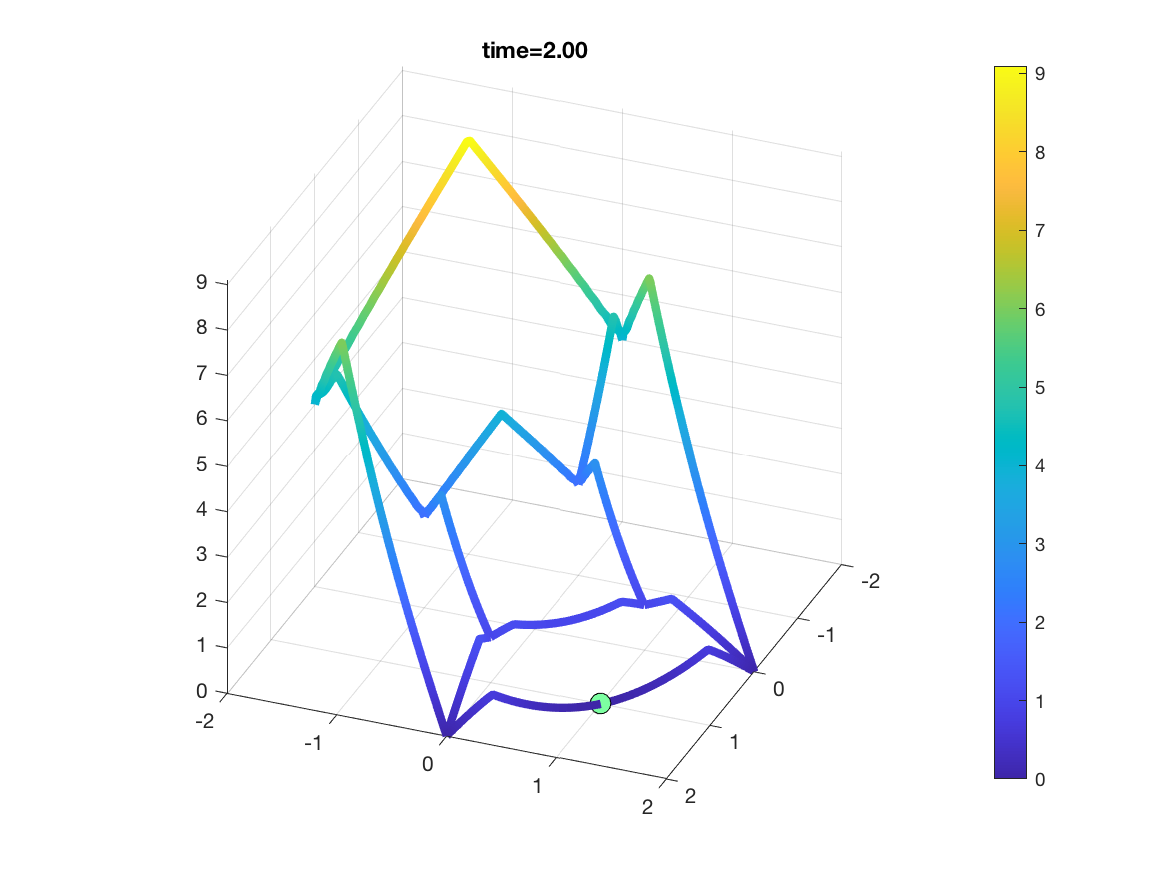}\includegraphics[width=0.5\textwidth]{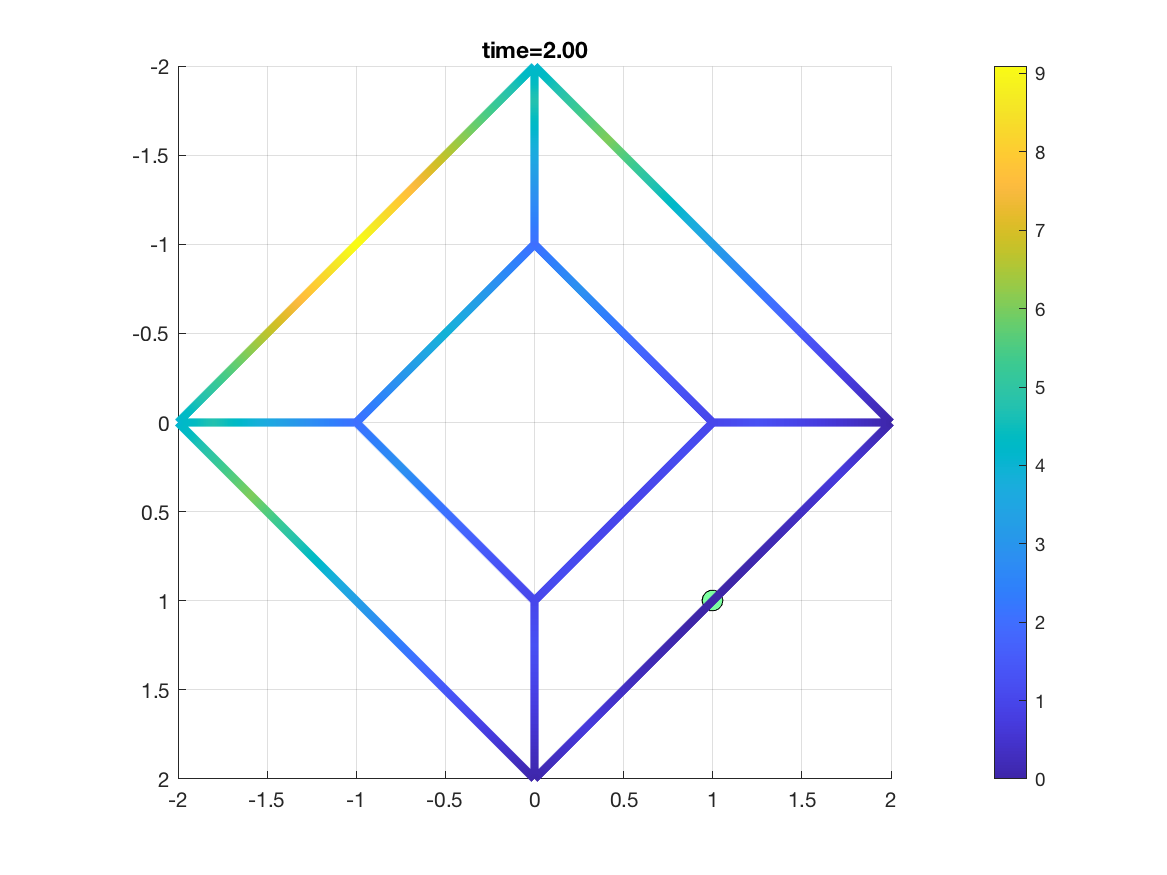}
	\centering
	\includegraphics[width=0.5\textwidth]{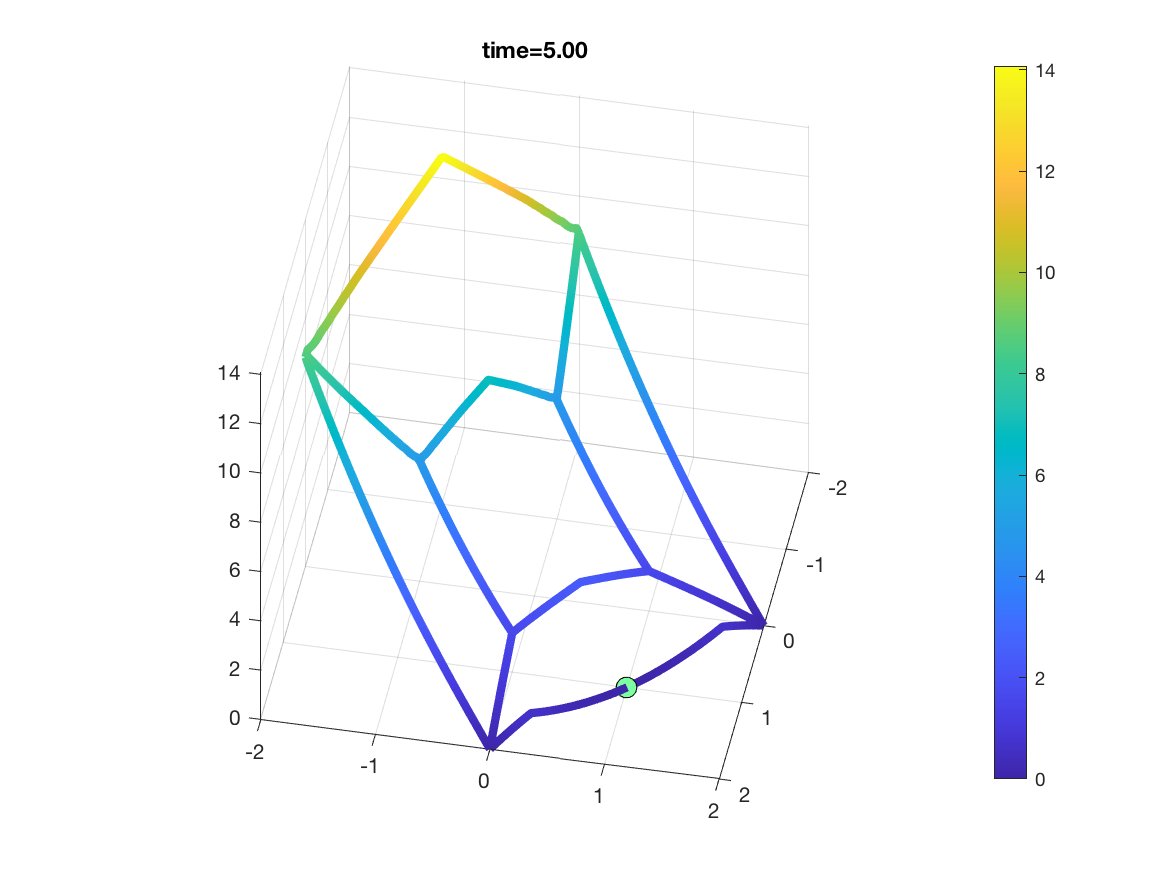}\includegraphics[width=0.5\textwidth]{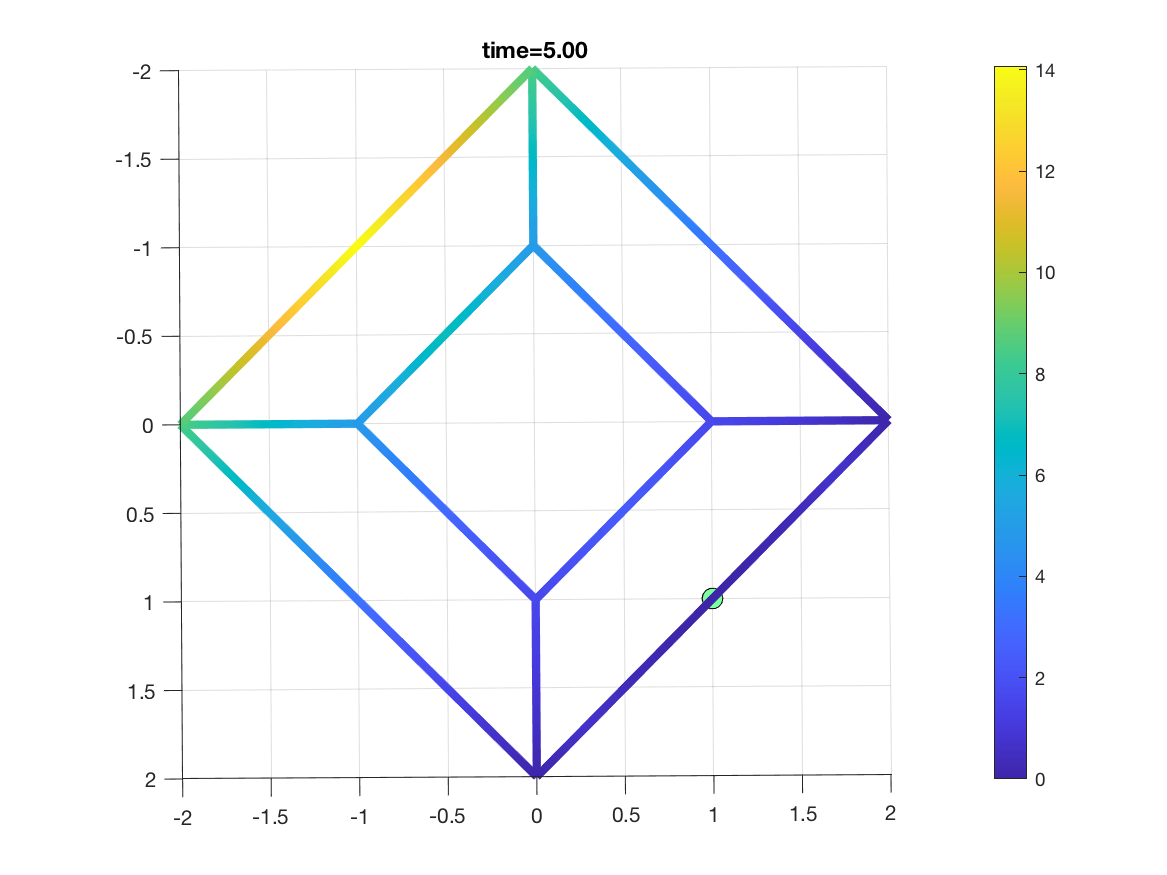}

\caption{Example in Section \ref{sec:Test2} with continuous Hamiltonian \label{fig:test2a} at time $T=2$ and $T=5$ \label{fig:test2}}
\end{figure}

 \begin{figure}[htbp]
\centering
	\includegraphics[width=0.5\textwidth]{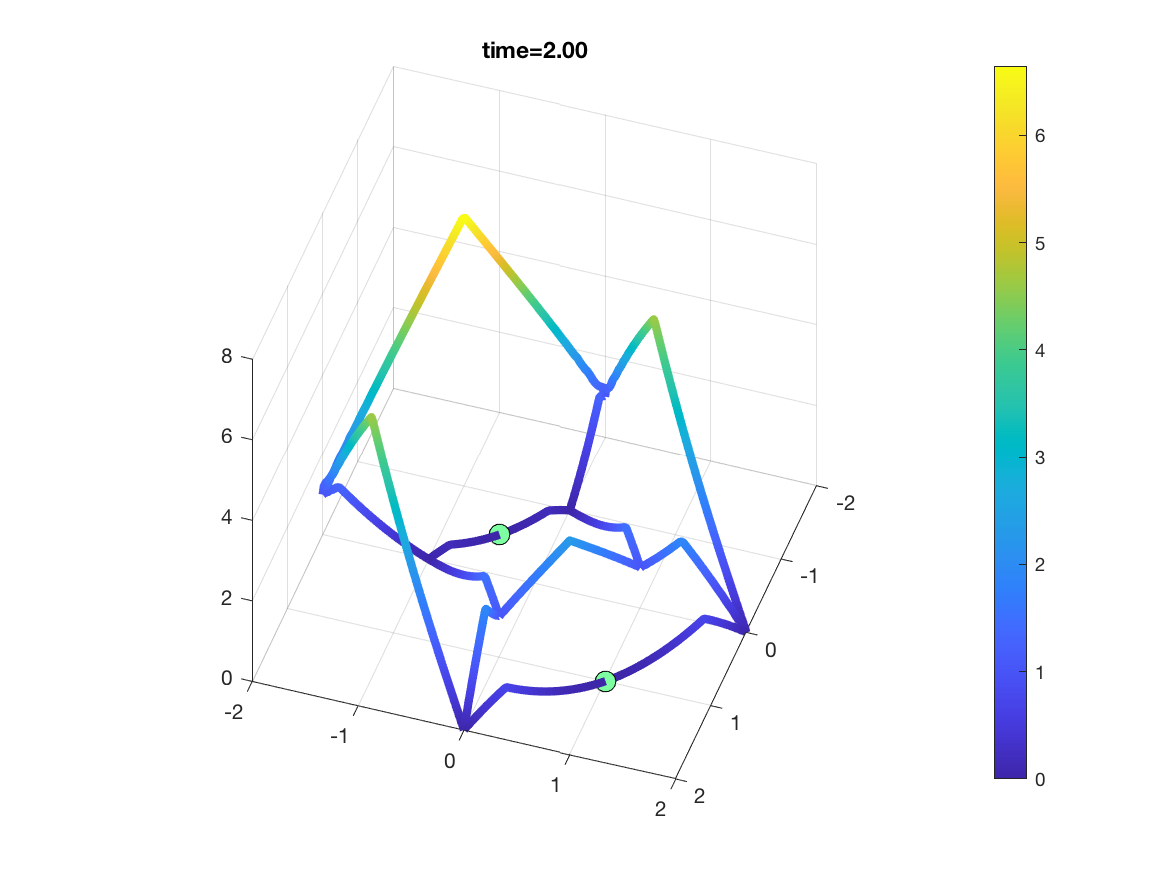}\includegraphics[width=0.5\textwidth]{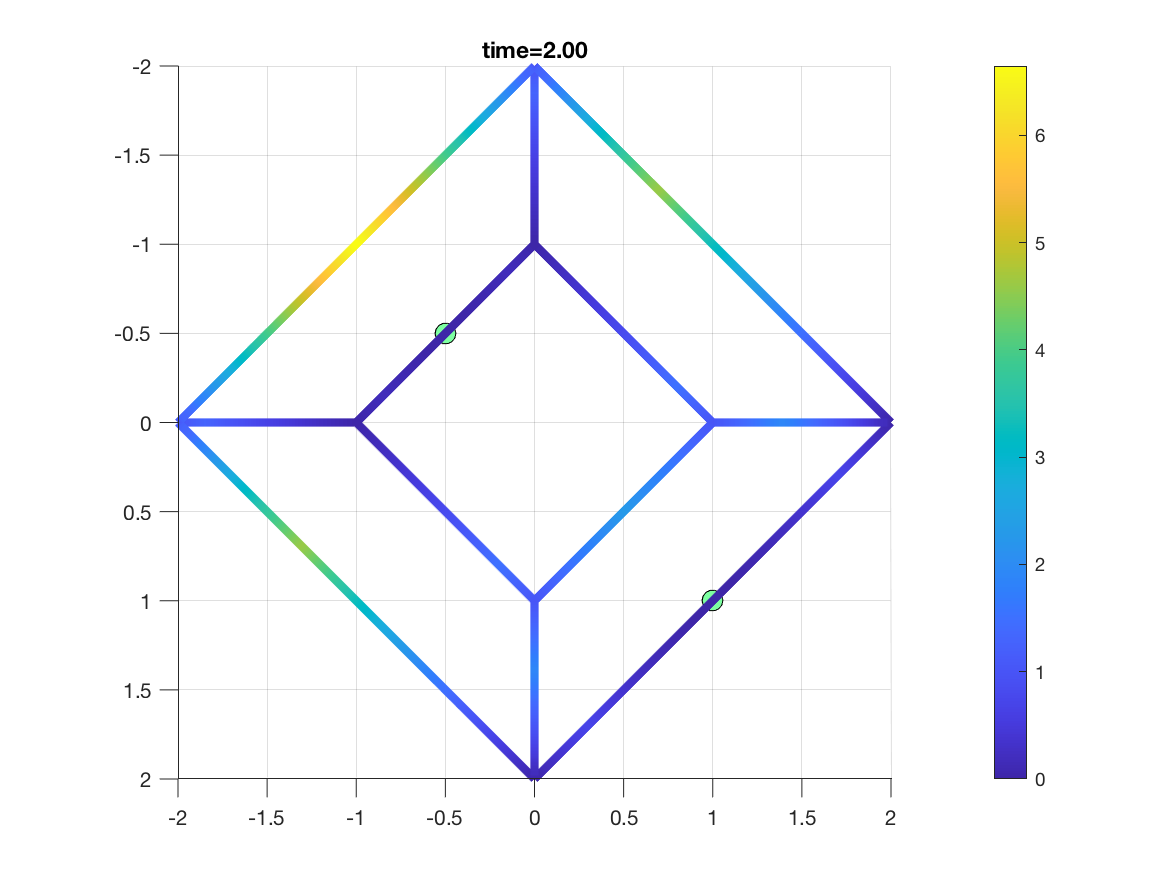}
		\includegraphics[width=0.5\textwidth]{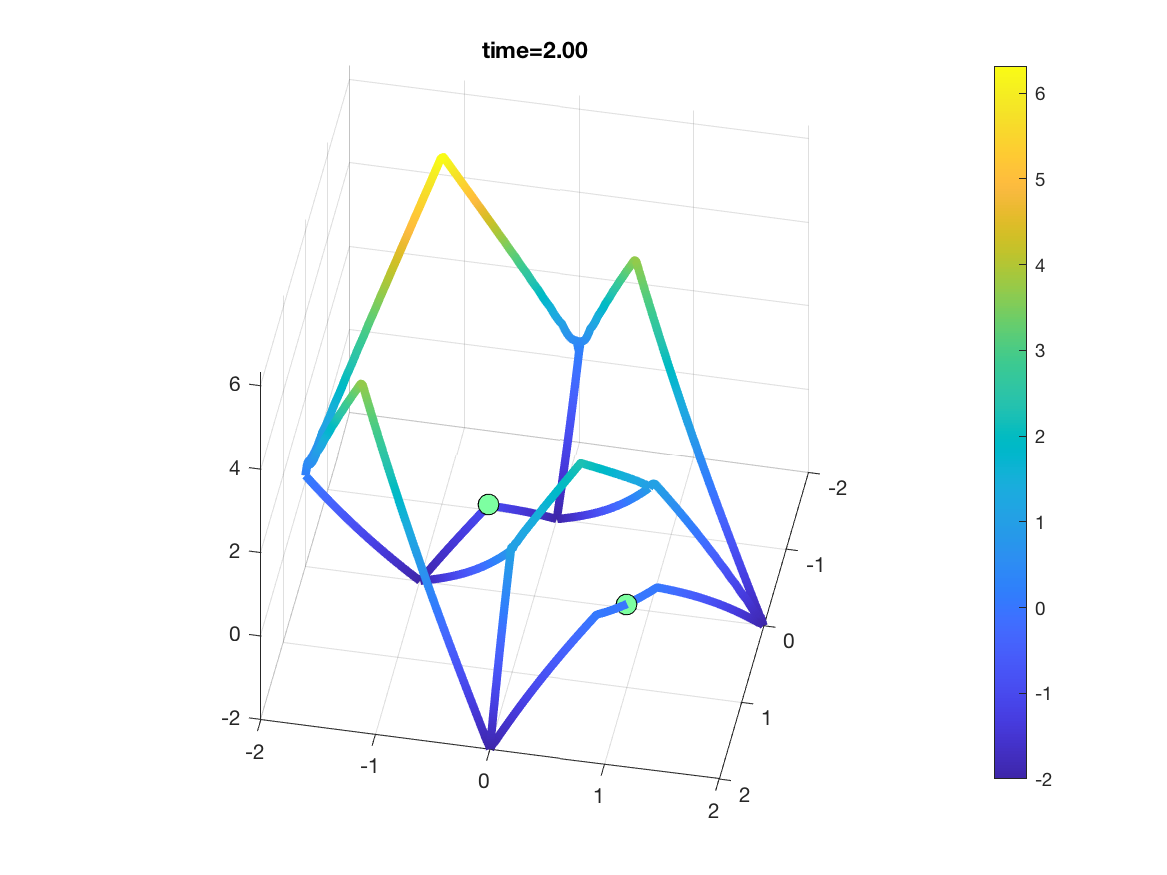}\includegraphics[width=0.5\textwidth]{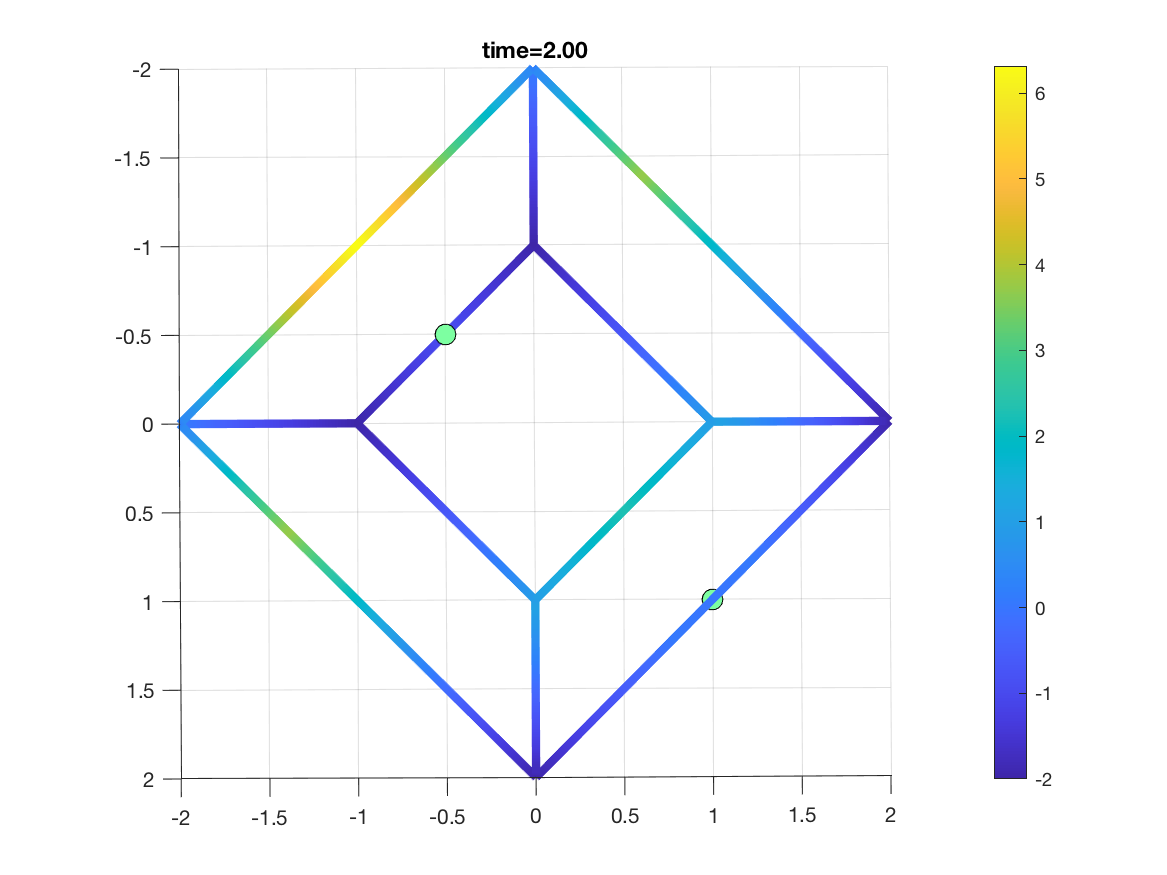}
\caption{Example in Section \ref{sec:Test2} with discontinuous Hamiltonian \label{fig:test2b} at time $T=2$ and flux limiters $c_2=c_3=c_4=c_8=0$ (top), and $c_2=c_3=c_4=c_8=-1$(bottom) }
\end{figure}

\bigskip

\appendix

\section{Modification of an Hamiltonian/Lagrangian}\label{penalized}

 We are given an Hamiltonian $\wtd H:[0,a] \times \R$, for some $a >0$,  satisfying assumptions {\bf(H1)}--{\bf(H4)}, namely $\wtd H$ is  continuous in both arguments, convex   and  superlinear  in the momentum variable, uniformly  in $s$, and finally the function $s \mapsto \wtd H(s,\mu)$ is  Lipschitz  continuous in $[0,a]$ for any $\mu \in \R$.   Due to superlinearity condition, a finite Lagrangian $\wtd L$ can be defined in $[0,a] \times \R$ through Fenchel transform
 \[\wtd L(x,\la)= \max_\mu \mu \, \la - \wtd H(x,\mu).\]

For any   compact   interval  $I$, we show in this appendix that $\wtd H$, and consequently by duality $\wtd L$,  can be modified  in order to obtain  a new  pair of Hamiltonian/Lagrangian  $ H$/$ L$  with the property that
\[  H(s,\mu)= \wtd H(s,\mu) \qquad\hbox{for $(s,\mu) \in [0,a] \times I$,}\]
 $ H$ is   linear at infinity with respect to the momentum variable, and  consequently $ L$ is   infinite outside a given compact of the variable $\la$.

As already pointed out in Section \ref{assu}, the  assumptions {\bf(H1)}--{\bf(H4)} imply that $\wtd H$ is locally Lipschitz continuous in $[0,a] \times \R$.
We denote by $\beta$ the  Lipschitz constant of $\wtd H$ in  $[0,a] \times \wtd I$, where $\wtd I$ is a compact neighborhood of $I$,  therefore
\begin{equation}\label{pena1}
 |b| \leq \beta   \qquad\hbox{for any $(s,\mu) \in [0,a] \times I$, $b \in \partial_\mu \wtd H(s,\mu)$,}
\end{equation}
where $\partial_\mu \wtd H(s,\mu)$ stands for the subdifferential of $\wtd H(s,\cdot)$ at $\mu$. Since $\wtd H$ is superlinear in the second argument, we can  fix $\beta_0 > \beta$ and  select $\mu_0 > 0$ such that
\begin{eqnarray*}
  \wtd H(s,\mu) &>& \wtd H(s, -\mu_0) - \be_0 (\mu +\mu_0)  \qquad\hbox{ for any $s \in [0,a]$, $\mu \in (-\infty , -\mu_0)$}\\
  \wtd H(s,\mu) &>& \wtd H(s, \mu_0) + \be_0 (\mu-\mu_0)  \qquad\hbox{ for any $s \in [0,a]$, $\mu \in (\mu_0, +\infty)$.}
\end{eqnarray*}
Therefore
\begin{eqnarray*}
  - \be_1(s) &\in& \partial_\mu \wtd H(s,-\mu_0)  \qquad\hbox{for any $s \in [0,a]$, some $\be_1(s) \geq \be_0$} \\
  \be_2(s)  &\in& \partial_\mu \wtd H(s,\mu_0) \qquad\hbox{for any $s \in [0,a]$, some $\be_2(s) \geq \be_0$.}
\end{eqnarray*}
This in turn implies by \eqref{pena1} that $\pm \mu_0 \not\in I$ and $[-\mu_0,\mu_0] \supset I$. We proceed defining
\[  H_0(s,\mu) = \left \{\begin{array}{cc}
                               \wtd H(s,\mu) & \quad\hbox{for any $s \in [0,a]$, $\mu \in [-\mu_0 , \mu_0]$} \\
                               \wtd H(s, -\mu_0) + \be_0 (\mu_0-\mu)  &\quad\hbox{ for any $s \in [0,a]$, $\mu \in (-\infty , -\mu_0)$} \\
                               \wtd H(s, \mu_0) + \be_0 (\mu-\mu_0) & \quad\hbox{ for any $s \in [0,a]$, $\mu \in (\mu_0, + \infty)$}
                             \end{array} \right .\]
Bearing in mind that $\wtd H$ is locally Lipschitz continuous in $[0,a] \times \R$, we see from the above definition that $ H_0$ is Lipschitz continuous in $[0,a] \times \R$.

We  finally define $ H(s,\cdot)$ as the convex envelope of $ H_0(s,\cdot)$, for $s \in  [0,a]$. We denote by $  L(s,\la)$  the convex conjugate of $ H(s,\mu)$.  Since $ H(s,\cdot)$ is continuous for any $s$, we have, see \cite{Rockafellar} Theorem 12.2
\begin{figure}[htbp]
\centering
	\includegraphics[width=0.6\textwidth]{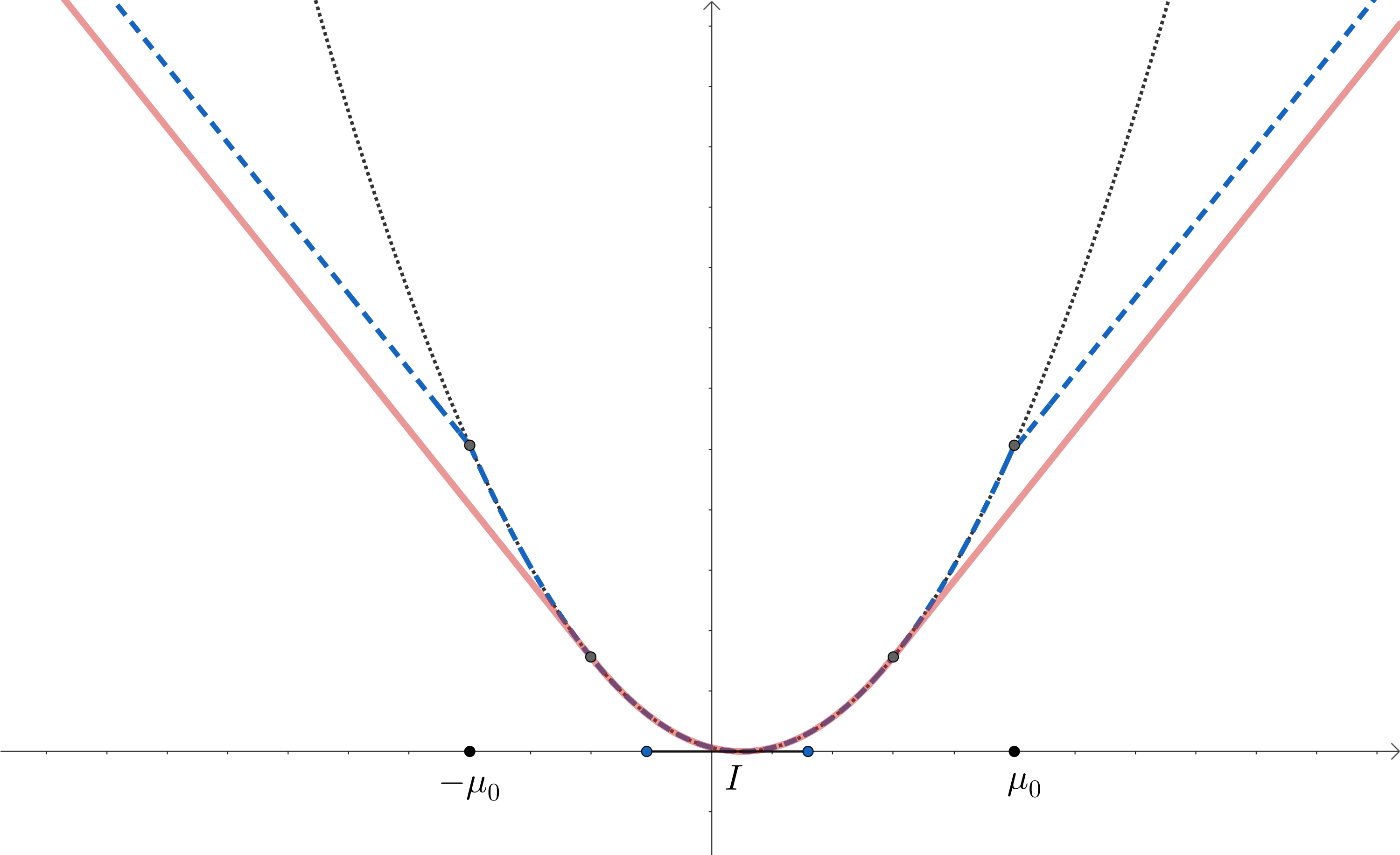}
\caption{$\wtd H(s,\cdot)$ (black dotted-line),  $ H_0(s,\cdot)$ (blue dashed-line), $ H(s,\cdot)$ (red solid line) }
 \label{fig:convessificata}
\end{figure}
\begin{Lemma} The convex conjugate of $ L(s,\cdot)$ is $ H(s,\cdot)$ for any $s \in [0,a]$.
\end{Lemma}

\smallskip
We see directly from the above construction that
\begin{equation}\label{pena0}
  H(s,\mu) \leq \wtd H(s,\mu) \qquad\hbox{and} \quad  L(s,\la) \geq \wtd L(s,\la) \qquad\hbox{for any $\la$, $\mu$, $s$.}
\end{equation}
\smallskip

\begin{Lemma}\label{toy} We have
\begin{itemize}
  \item[{\bf (i)}] $ H(s,\mu)= \wtd H(s,\mu)$ in $[0,a] \times I$;
  \item[{\bf (ii)}] $ L(s,\la) \equiv + \infty$ in $[0,a] \times \big ( \R \setminus [-\be_0,\be_0] \big )$;
  \item[{\bf (iii)}] $ L$ is Lipschitz continuous in $[0,a] \times [-\be_0,\be_0]$.
\end{itemize}
\end{Lemma}
\begin{proof}  We fix $s \in [0,a]$.  To prove {\bf (i)},  we first claim that if $\mu_* \in I$ then the line
\[(\mu, b \, (\mu - \mu_*) +  H_0(s,\mu_*)) \qquad \mu \in \R\]
stays below the graph of $ H_0(s, \cdot)$ for any $b$ in the subdifferential of $\wtd H(s,\cdot)$ at $\mu_*$. If not, due to the convex character of $\wtd H(s,\cdot)$, a line of the above type should cross the graph of $ H_0(s, \cdot)$ at points  outside $[-\mu_0,\mu_0]$ where $ H_0(s, \cdot)$ is linear with angular coefficient $\be_0$ or $-\be_0$. This implies that $|b| > \be_0$, which is in contradiction with the Lipschitz constant of $\wtd H(s,\cdot)$ in $I$ being strictly less or equal to $\be_0$, see \eqref{pena1}.

Bearing in mind that the $ H(s,\cdot)$ is the pointwise maximum of the  convex functions not exceeding $ H_0(s,\cdot)$, we deduce that
\[ H(s,\mu) \geq b \, (\mu - \mu_*) +  H_0(s,\mu_*)\]
which yields
\[ H(s,\mu_*) \geq  H_0(s,\mu_*) = \wtd H(s,\mu_*)\]
and since the opposite inequality is trivially true, we end up with
\[  H(s,\mu_*) = \wtd H(s,\mu_*),\]
as was claimed in item {\bf (i)}.
To prove {\bf (ii)}, we take into account that the Fenchel transform of $ H_0(s,\cdot)$ and $ H(s,\cdot)$ coincide, see \cite{Hiriart} Corollary 1.3.6.  By taking $\la > \be_0$, we  get
\begin{eqnarray*}
   L (s,\la) &\geq& \sup_{\mu \in (\mu_0,+\infty)}\mu \, \la -  H_0(s,\mu) \\
   &=& \sup_{\mu \in (\mu_0,+\infty)} \mu \, (\la - \be_0 ) + \be_0 \, \mu_0 -  \wtd H(s,\mu_0) = + \infty.
\end{eqnarray*}
If instead $\la < -\be_0$, we have
\begin{eqnarray*}
   L (s,\la) &\geq& \sup_{\mu \in (-\infty, -\mu_0)}\mu \, \la -  H_0(s,\mu) \\
   &=& \sup_{\mu \in (-\infty, -\mu_0)} \mu \, (\la + \be_0 ) - \be_0 \, \mu_0 -  \wtd H(s,\mu_0) = + \infty,
\end{eqnarray*}
which completes the proof of {\bf (ii)}.  We proceed pointing out  that for  $\la \in [- \be_0,\be_0]$ the function
\[ \mu \mapsto \mu \, \la -  H_0(s,\mu) = \mu \, (\la - \be_0 ) + \be_0 \, \mu_0 -  \wtd H(s,\mu_0) \]
is decreasing for $\mu$ varying in $[\mu_0, + \infty)$ and the same property holds true for $ \mu \in (- \infty,, -\mu_0]$. this implies that
\begin{equation}\label{pena2}
 L(s,\la)= \max_{\mu \in [ -\mu_0,\mu_0]} \mu \, \la -  H_0(s,\mu) \quad\hbox{for $\la \in [-\be_0,\be_0]$.}
 \end{equation}
 Given $s_1$, $s_2$ in $[0,a]$, $\la \in [-\be_0,\be_0]$, we have
 \[ L(s_1, \la) -  L(s_2,\la) \leq \mu_1 \, \la -  H_0(s_1,\mu_1) - \mu_1 \, \la +  H_0(s_2,\mu_1),\]
 where $\mu_1$ is an optimal element for $ L(s_1, \la)$. We derive
\begin{equation}\label{pena3}
   L(s_1, \la) -  L(s_2,\la) \leq \ell \, |s_1 - s_2|,
 \end{equation}
 where $\ell$  is a Lipschitz constant for $ H_0$ in $[0,a] \times \R$. Given $s \in [0,a]$, $\la_1$, $\la_2$  in $[-\be_0,\be_0]$, we have
\[ L(s, \la_1) - L(s,\la_2) \leq \mu_1 \, \la_1 -  H_0(s,\mu_1) - \mu_1 \, \la_2 +  H_0(s,\mu_2),\]
 where $\mu_1$ is an optimal element for $ L(s, \la_1)$. Taking into account \eqref{pena2} we get
 \begin{equation}\label{pena4}
   L(s, \la_1) - L(s,\la_2) \leq  \mu_0 \, |\la_1 - \la_2|
 \end{equation}
 Inequalities \eqref{pena3}, \eqref{pena4} show item {\bf (iii)}.
\end{proof}
\bigskip

\section{Spaces convergence}

In this section we present in a abstract frame some convergence results we use in the paper. We are given a  metric space $Y$ with distance denoted by $d$.

\begin{Definition}\label{abspaces} Given a sequence  of closed  subspaces $Y_m \subset Y$, we say that  $Y_m$ converges to $Y$ as $m$ goes to infinity, mathematically $Y_m \to Y$, if for any $y \in Y$ there exists a sequence $y_m \in Y_m$ with
\[y_m \to y.\]
\end{Definition}

\smallskip

\begin{Definition}\label{deftoy} Assume that $Y_m \to Y$ in the sense of the above definition. Given a sequence of functions $v_m: Y_m \to \R$, $v:Y \to \R$, we say that $v_m$ {\em  uniformly converges} to $v$ when $Y_m \to Y$  if
\[\lim v_m(y_m)= v(y) \qquad\hbox{for any $y \in Y$, $y_m \in Y_m$ with $y_m \to y$.}\]
\end{Definition}

\smallskip
\begin{Remark}\label{prepremax}  Assume that $v_m$ {\em  uniformly converges} to $v$ when $Y_m \to Y$. The above definition implies that if  a subsequence $z_{m_n} \in Y_{m_n}$ converges to $y \in Y$, then
\[\lim v_{m_n}(z_{m_n})= v(y).\]
In fact, since $Y_m \to Y$ there is a sequence $y_m \in Y_m$ converging to $y$, we define
\[ \wtd y_m = \left \{ \begin{array}{cc}
                y_m  & \quad\hbox{if $m \neq m_n$ for all $n$} \\
                z_{m_n} & \quad\hbox{if $m = m_n$ for all $n$}
              \end{array} \right . \]
It is clear that $\wtd y_m \to y$ and consequently $v_m(\wtd y_m) \to v(y)$ so that
$v_{m_n}(z_{m_n})$, which is a subsequence of $v_m(\wtd y_m)$ converges to $v(y)$ as well.
\end{Remark}
\smallskip
\begin{Proposition}\label{premax} Assume that the  sequence $v_m: Y_m \to \R$  uniformly converges to a function  $v: Y \to \R$, then $v$ is continuous.
\end{Proposition}
\begin{proof}    Let us fix  $y_0 \in Y$ and a
sequence $z_n$
 converging to it. Taking into account that  $Y_m \to Y$ and  exploiting the   definition of  uniform limit,   we can  select  indices  $m_n$, with $m_n \to + \infty$, such that $y_{m_n} \in Y_{m_n}$ and
\[ d(y_{m_n}, z_n)  < \frac 1n,  \quad |v(z_n)
- v_{m_n}(y_{m_n})| < \frac 1n. \]
This implies by   Remark \ref{prepremax} that
\[ \lim_n v(z_n) = v(y_0),\]
 as desired.
 \end{proof}

\bigskip

The stability of maxima and maximizers (resp. minima and minimizers)
is one the main properties of the above introduced
convergence.

\medskip

\begin{Proposition}\label{max} Let $U$ be  a  compact subset of $Y$  with $Y_m \cap U \to U$, and   $v_m:  Y_m \cap U \to \R$  a sequence  uniformly converging to a function $v: U  \to \R$.
If  $y_m$ is a  sequence of  maximizers (resp. minimizers) of $v_m$ in $U \cap Y_m$, then any  of its limit point is a maximizer (resp. minimizer) of  $v$ in $U$. In addition
\[ \max_{U \cap Y_m} v_m \to \max_U v\]
\[ \hbox {(resp. $\min_{U \cap Y_m} v_m \to \min_U v$)}\]
\end{Proposition}
\begin{proof}  It is enough to prove the assertion for maximizers and maxima, the corresponding one for minimizers and minima can obtained arguing in the same way. We denote by  $y_0 \in U$ a limit point of the $y_m$'s, we assume more precisely that a subsequence $y_{m_n}$ converges to $y_0$.
 We have, by the very definition of  uniform limit  and Remark \ref{prepremax}, that
\begin{equation}\label{stab1}
   \lim_n \max_{U \cap Y_{m_n}} v_{m_n} = \lim v_{m_n}(y_{m_n}) = v(y_0)
\leq \max_U v.
\end{equation}
 On the other side, a maximizer  $z_0$ of $v$ in $U$ does exist because $U$ is compact and $v$ continuous in $U$ by Proposition \ref{premax}. Thanks to  the condition $Y_m \cap U \to U$, we further find  $z_m \in Y_m \cap U$   converging to $z_0$  and so satisfying
 \[\lim_m v_m(z_m)= v(z_0).\]
 We have
 \begin{equation}\label{stab2}
   \max_U v = v(z_0)= \lim_n v_m(z_m) \leq \liminf_m \max_{U \cap Y_m} v_m.
\end{equation}
 The assertion  follows combining \eqref{stab1} and
 \eqref{stab2}.
 \end{proof}

\bigskip

\section{How to read figures}\label{readfigure}
Let us explain how to read the  figures in all the paper.  As an example we consider the two plots in the top of Figure \ref{fig:test2}.
Let us start with the plot in the top left.
\begin{enumerate} 
\item The network is embedded in $\R^3$ and the three-dimensional box drawn $[-2,2]\times [-2,2]\times[-2,9]$ helps a three-dimensional visualization.

\item Accordingly, a 3d view of  the  solution $v_{\Delta}(\cdot,T)$ has been displayed in the box, together with the target point $(\bar x, v_{\Delta}(\bar x,T))$, indicated as a green circle.
\item The color  indicates the value $v_{\Delta}(\cdot,T)$ according to the color scale  on the right. 
\item The viewpoint has been chosen in such a way that one sees the larger values of $v_{\Delta}(x)$ in the background, and the smaller ones in the foreground.
\item Since it is not easy to detect all the local minima of $v_{\Delta}(\cdot)$ in the figure, we have put in evidence two of them, vertices $(-2,0)$ and $(-1,0)$, in the zoomed-in Figure \ref{fig:zoom}.
 \end{enumerate}
Next, we explain the  top plot in the right of Figure \ref{fig:test2}.
\begin{enumerate} 
\item The picture is two dimensional as indicated by the 2d box $[-2,2]\times [-2,2]$. 
\item The drawn function is the projection of the solution $v_{\Delta}(\cdot,T)$ in the $x_1x_2$ plane.
\item  The color  indicates the value $v_{\Delta}(\cdot,T)$ according to the color scale  on the right.
\item  The target point $\bar x$  is visualised as a green circle.

 \end{enumerate}
\begin{figure}[htbp]
\centering
		\includegraphics[width=0.6\textwidth]{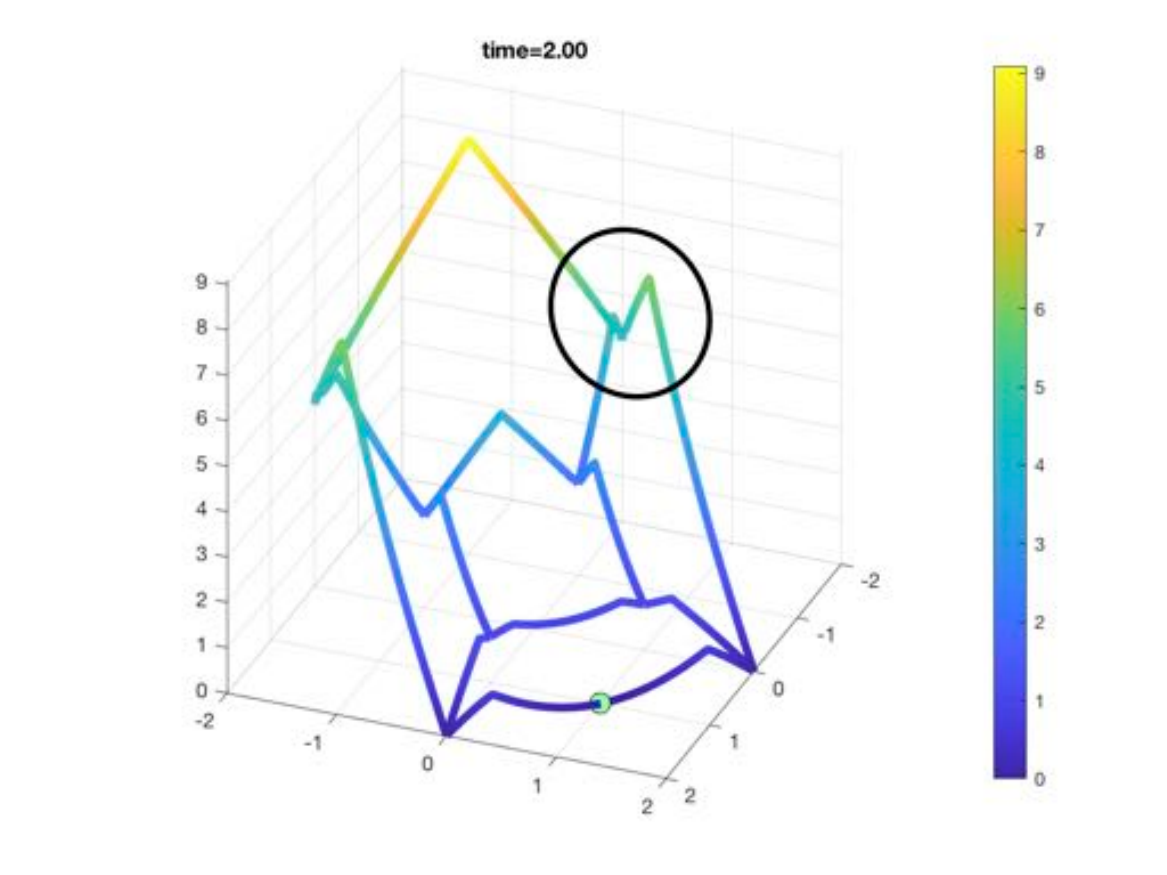}\includegraphics[width=0.6\textwidth]{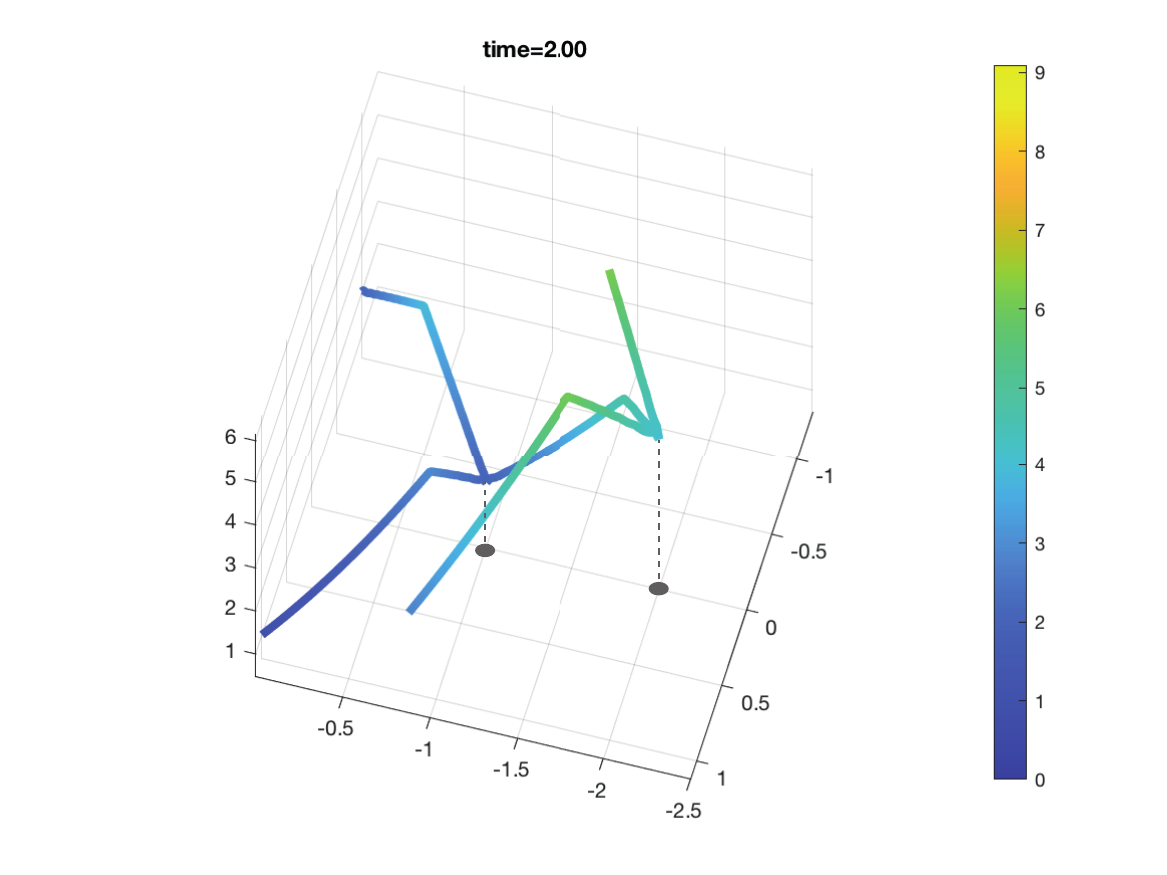}
\caption{ 
The figure on the right is a zoomed-in display of the region shown in the black circle of the one of the left. We have also made a rotation on the right  in order to better show the minima.
 \label{fig:zoom}}
\end{figure} 
\bigskip


\bibliography{sn-bibliography}

\begin{thebibliography}{10}
\bibitem{ACCT13}
{\sc Y. Achdou, F. Camilli, A. Cutri, N. Tchou}, { \em Hamilton-Jacobi equations constrained on networks},
 Nonlinear Differential Equations and Applications NoDEA, 20 (2013), pp. 413--445.


 \bibitem {Buttazzo}
{\sc G. Buttazzo, M. Giaquinta, S. Hildebrandt}, {\em One-dimensional variational problems},
Oxford Lecture Series in Mathematics and its Applications,
15, The Clarendon Press, Oxford University Press, New York, (1998).

\bibitem{carlini13}
{\sc  E.~Carlini, M.~Falcone,  A.~Festa},
{\em A brief survey on semi--lagrangian schemes for image processing}, Mathematics and Visualization, 202529 (2013),pp.~191--218.
\bibitem{camilli2013approximation}
{\sc F.~Camilli, A.~Festa,  D.~Schieborn}, {\em An approximation scheme for
  a hamilton--jacobi equation defined on a network}, Applied Numerical
  Mathematics, 73 (2013), pp.~33--47.
\bibitem{CCM18}
{\sc  F.~Camilli, E. Carlini, C. Marchi}, {\em A
flame propagation model on a network with
application to a blocking problem}, Discrete Continuous Dynamical-Systems 11(5) (2018),
pp. 825--843.
\bibitem{CMS13}
{\sc F.~Camilli, C.~Marchi, D. Schieborn},  {\em The vanishing viscosity limit for Hamilton–Jacobi equations
on Networks}. J. Differ. Equ. 254, (2013) pp.~4122–4143.
  \bibitem{carlini20}
  {\sc  E. ~Carlini, N.Forcadel,  A.~Festa},
{\em A Semi-Lagrangian scheme for Hamilton-Jacobi-Bellman equations on networks},
SIAM J. Numer. Anal. 58, (2020) pp.~3165-3196, . 

\bibitem{Clarke75}
{\sc F.~H.~Clarke},
  {\em The {E}uler-{L}agrange differential inclusion},
  J. Differential Equations, 19, (1975), pp.~80--90.

  \bibitem{costeseque2015convergent}
{\sc G.~Costeseque, J.-P. Lebacque, R.~Monneau}, {\em A convergent scheme
  for hamilton--jacobi equations on a junction: application to traffic},
  Numerische Mathematik, 129 (2015), pp.~405--447.
  \bibitem{CrandallLions84}
{\sc M.G.~Crandall,  P.-L. Lions},
   { \em Two approximations of solutions of {H}amilton-{J}acobi
    equations}, Math. Comp., 43 (1984), pp.~1--19.
  \bibitem{FF15} M. Falcone, R. Ferretti, {\em Semi-Lagrangian approximation schemes for linear and
Hamilton-Jacobi equations}, vol. 133, SIAM (2014)
  \bibitem{GK19} {\sc  J. Guerand, M. Koumaiha}, {\em  Error estimates for finite difference schemes associated with
hamilton-jacobi equations on a junction}, Numerische Mathematik, 142(3) (2019), pp.~525--575.

\bibitem {Hiriart}
{\sc J--B Hiriart--Urruty,  C. Lemar\'{e}chal}, {\em Fundatnentals of Convex Analysis },
Grundlehren Text Editions, Springer-Verlag, Berlin Heidelberg  (2004).

  \bibitem{ImbertMonneau17}
{\sc C. Imbert, R. Monneau}, {\em Flux-limited solutions for quasi-convex {H}amilton-{J}acobi equations
on networks}, Annales Scientiques de l'ENS, (2017), pp.~357--448.

\bibitem{ImbertMonneauZidani2013} 
{\sc C. Imbert, R. Monneau, H. Zidani},
{\em A {H}amilton-{J}acobi approach to junction problems and
              application to traffic flows},
              ESAIM Control Optim. Calc. Var, 19, (2013),pp. 129--166.
    
    \bibitem{PozzaSiconolfi}
    {\sc  M.~Pozza, A.~Siconolfi},	
   { \em Lax--Oleinik formula on networks}, SIAM J. Math. Anal.,3, (2023), pp.~2211--2237.

   \bibitem {Rockafellar}
{\sc R. T. Rockafellar}, {\em Convex analysis},
Princeton UNiversity Press,
Princeton, New Jersey (1970).

\bibitem{Schieborn13} {\sc D.Schieborn, F. Camilli},
 { \em Viscosity solutions of {E}ikonal equations on topological networks},
Calc. Var. Partial Differential Equations,
 46, (2013),pp. {671--686},
      
       \bibitem{Siconolfi}
    {\sc   A.~Siconolfi}, {\em Time-dependent Hamilton-Jacobi equations on
networks} , J. Math. Pures Appl. (9) 163 (2022),pp. ~702–738.
\end{thebibliography}

\end{document}